\documentclass[10pt]{smfart}

\RequirePackage[T1]{fontenc}
\RequirePackage{amsfonts,latexsym,amssymb}

\RequirePackage{smfenum}
\RequirePackage[frenchb]{babel}
\addto\extrasfrenchb{\bbl@nonfrenchitemize
\bbl@nonfrenchspacing}

\RequirePackage{mathrsfs}
\let\mathcal\mathscr

\usepackage[matrix,arrow]{xy}
\usepackage{url}
\usepackage{accents}
\usepackage{enumitem}
\usepackage{fge}
\usepackage{bbm}

\theoremstyle{plain}
\numberwithin{equation}{section}
\newtheorem{prop}[equation]{\propname}
\newtheorem{theo}[equation]{\theoname}

\newtheorem{coro}[equation]{\coroname}

\newtheorem{lemm}[equation]{\lemmname}
\theoremstyle{definition}
\theoremstyle{remark}
 
\newtheorem{rema}[equation]{\remaname}

\def\paskunas{Pa\v{s}k\={u}nas}

\let\cal\mathcal
\let\goth\mathfrak

\def\Q{{\bf Q}} \def\Z{{\bf Z}}

\def\N{{\bf N}}
\def\O{{\cal O}}
\def\G{{\cal G}}

\def\dual{{\boldsymbol *}}

\def\bPi{{\boldsymbol\Pi}}

\def\Qbar{\overline{\bf Q}}
\def\epsilon{\varepsilon}

\def\piqp{{\bf P}^1}

\def\matrice#1#2#3#4{{\big(\begin{smallmatrix}#1&#2\\ #3&#4\end{smallmatrix}\big)}}

\def\wotimes{\,\widehat\otimes\,}

\newcommand{\gl}{{\rm GL}}

\def\invlim{\mathop{\vtop{\ialign{##\crcr$\hfill{\lim}\hfil$\crcr
\noalign{\kern1pt\nointerlineskip}\leftarrowfill\crcr\noalign
{\kern -3pt}}}}\limits}
\def\dirlim{\mathop{\vtop{\ialign{##\crcr$\hfill{\lim}\hfil$\crcr
\noalign{\kern1pt\nointerlineskip}\rightarrowfill\crcr\noalign
{\kern -3pt}}}}\limits}

\def\epsilon{\varepsilon}
\let\mathcal\mathscr

\numberwithin{equation}{section}

\begin{document}
\title[Anneaux de Kisin]
{Correspondance de Langlands locale $p$-adique et anneaux de Kisin}
\author{Pierre Colmez}
\address{CNRS, IMJ-PRG, Sorbonne Universit\'e, 4 place Jussieu, 75005 Paris, France}
\email{pierre.colmez@imj-prg.fr}
\author{Gabriel Dospinescu}
\address{CNRS, UMPA, \'Ecole Normale Sup\'erieure de Lyon, 46 all\'ee d'Italie, 69007 Lyon, France}
\email{gabriel.dospinescu@ens-lyon.fr}
\author{Wies{\l}awa Nizio{\l}}
\address{CNRS, IMJ-PRG, Sorbonne Universit\'e, 4 place Jussieu, 75005 Paris, France}
\email{wieslawa.niziol@imj-prg.fr}
 \date{\today}
\thanks{Les trois auteurs sont membres du
projet ANR-19-CE40-0015-02 COLOSS.}
\begin{abstract}
Nous utilisons un proc\'ed\'e de compl\'etion ${\cal B}$-adique
et la correspondance de Langlands locale $p$-adique pour $\gl_2(\Q_p)$
pour donner une construction
des anneaux de Kisin et des repr\'esentations galoisiennes
universelles associ\'ees (en dimension~$2$ et pour $\Q_p$) \`a partir de la correspondance
de Langlands locale classique.
Cela fournit, en particulier, une preuve uniforme 
de la conjecture de Breuil-M\'ezard (version g\'eom\'etrique)
dans le cas supercuspidal.
\end{abstract}
\begin{altabstract}
We use a ${\cal B}$-adic completion and the $p$-adic local Langlands correspondence for $\gl_2(\Q_p)$
to give a construction of Kisin's rings and the attached universal Galois representations
(in dimension $2$ and for $\Q_p$) directly from the classical Langlands correspondence.
This gives, in particular, a uniform proof of the geometric Breuil-M\'ezard conjecture in
the supercuspidal case.
\end{altabstract}

\setcounter{tocdepth}{1}

\maketitle

{\Small
\tableofcontents
}

\section*{Introduction}
Nous donnons, pour $\G_{\Q_p}$ (groupe de Galois absolu de $\Q_p$) et en dimension~$2$,
une construction directe des anneaux de Kisin et des repr\'esentations universelles attach\'ees.

Plus pr\'ecis\'ement, soit $L$ une extension finie de $\Q_p$ d'anneau des entiers $\O_L$ et de corps
r\'esiduel $k_L$, et soient $M$ un $(\varphi,N,\G_{\Q_p})$-module de rang $2$ irr\'eductible,
$a<b$ des entiers et $\overline\rho:\G_{\Q_p}\to \gl_2(k_L)$ 
une repr\'esentation continue, semi-simple.
Notre construction
 part de la repr\'esentation ${\rm LL}(M)$ de $G:=\gl_2(\Q_p)$ associ\'ee
\`a $M$ par la recette de Fontaine et la correspondance de Langlands locale classique (c'est une
repr\'esentation supercuspidale puisque $M$ est suppos\'e irr\'eductible).
On suppose que 
la repr\'esentation localement
alg\'ebrique
$${\rm LL}^{[a,b]}(M)={\rm LL}(M)\otimes {\rm Sym}^{b-a-1}\otimes{\det}^a$$
a un caract\`ere central unitaire. Elle admet alors des $\O_L$-r\'eseaux stables
par $G$; on fixe un tel r\'eseau $\Lambda$.
La repr\'esentation $\overline\rho$ d\'etermine un bloc ${\cal B}$ 
de la cat\'egorie des $\O_L$-repr\'esentations de $G$
lisses de longueur finie et de caract\`ere central \'egal \`a celui de ${\rm LL}^{[a,b]}(M)$.
On d\'efinit\footnote{Cette construction appara\^{\i}t d\'ej\`a dans~\cite[(1.6.4)]{Kis2}.}
{\it le compl\'et\'e ${\cal B}$-adique} $\Lambda_{\cal B}$ 
de $\Lambda$ comme la limite projective
des quotients de longueur finie de $\Lambda$ dont toutes les composantes de Jordan-H\"older appartiennent
\`a ${\cal B}$, et on pose ${\rm LL}_{\cal B}^{[a,b]}(M)=L\otimes_{\O_L}\Lambda_{\cal B}$ (le r\'esultat
est ind\'ependant du choix de $\Lambda$). 
Soient alors\footnote{
Rappelons que l'on dispose d'un foncteur $\Pi\mapsto{\bf V}(\Pi)$ associant une $L$-repr\'esentation
de $\G_{\Q_p}$ \`a une $L$-repr\'esentation unitaire $\Pi$ de $G$, et que la correspondance de Langlands
locale $p$-adique 
$V\mapsto\bPi(V)$ associe \`a une $L$-repr\'esentation de $\G_{\Q_p}$, de dimension~$2$,
une $L$-repr\'esentation de $G$ v\'erifiant ${\bf V}(\bPi(V))=V$.}:
$$R_{M,\overline\rho}^{[a,b]}:={\rm End}_G({\rm LL}_{\cal B}^{[a,b]}(M)),\quad
X_{M,\overline\rho}^{[a,b]}:={\rm Spec}\,R_{M,\overline\rho}^{[a,b]},\quad
\rho_M^{[a,b]}:={\bf V}({\rm LL}_{\cal B}^{[a,b]}(M))$$
Notre r\'esultat principal est que $\rho_M^{[a,b]}$ est la famille universelle
des repr\'esentations de $\G_{\Q_p}$ de type $(M,a,b)$ dont la r\'eduction\footnote{Par d\'efinition,
on prend la semi-simplifi\'ee de la r\'eduction d'un r\'eseau stable par $\G_{\Q_p}$.} est $\overline\rho$, 
(i.e. de dimension~$2$,
potentiellement semi-stables \`a poids de Hodge-Tate $a$ et $b$, et dont le $D_{\rm pst}$
est isomorphe \`a $M$).

\begin{theo} \label{petit1}
{\rm (i)} $R_{M,\overline\rho}^{[a,b]}$ est un anneau commutatif s'identifiant
\`a l'anneau des fonctions analytiques born\'ees sur un ouvert analytique de $\piqp$; 
alg\'ebriquement,
c'est le produit d'un nombre fini d'anneaux principaux.

{\rm (ii)} $\rho_M^{[a,b]}$ est libre de rang~$2$ sur $R_{M,\overline\rho}^{[a,b]}$ et
l'application de sp\'ecialisation $x\mapsto\rho_x$ induit, pour toute extension
finie $L'$ de $L$, une bijection entre
$X_{M,\overline\rho}^{[a,b]}(L')$ et l'ensemble des $L'$-repr\'esentations de $\G_{\Q_p}$
de r\'eduction $\overline\rho$ et de type $(M,a,b)$.
\end{theo}
Ce th\'eor\`eme est prouv\'e aux \S\S\,\ref{PETIT13} et~\ref{PETIT14} (les notations y sont un peu
diff\'erentes car on met l'accent sur le bloc ${\cal B}$ plut\^ot que la repr\'esentation
$\overline\rho$, et donc $R_{M,\overline\rho}^{[a,b]}$ devient $R_{M,{\cal B}}^{[a,b]}$
et $\rho_M^{[a,b]}$ devient $\rho_{M,{\cal B}}^{[a,b]}$).
La preuve repose sur deux ingr\'edients:

$\bullet$ Le {\og th\'eor\`eme $R=T$ \fg} local de {\paskunas}~\cite{Paskext,PT} qui montre que 
${\bf V}({\rm LL}_{\cal B}^{[a,b]}(M))$ est {\og de type $\gl_2$ \fg} et munit
ce module d'une action de l'anneau des d\'eformations universelles du pseudo-caract\`ere
${\rm Tr}\circ\overline\rho$.

$\bullet$ Une extension des r\'esultats de~\cite{D-dR} selon lesquels,
 si $V$ est irr\'eductible de dimension~$2$, de Rham
\`a poids de Hodge-Tate distincts, et si $E$ est une extension de $V$ par~$V$, alors $\bPi(V)$
est la compl\'et\'ee de ses vecteurs localement alg\'ebriques si et seulement si $E$ est de Rham.
Ces r\'esultats permet de montrer que $R_{M,\overline\rho}^{[a,b]}$ est un produit d'anneau principaux
et que $\rho_M^{[a,b]}$ est libre de rang~$2$ sur $R_{M,\overline\rho}^{[a,b]}$
(cf.~th.\,\ref{petit15} et cor.\,\ref{petit16}). 

\begin{rema}
(i) Pour $\overline\rho$ g\'en\'erique, l'existence de
$\rho_M^{[a,b]}$ peut se d\'eduire des r\'esultats de
Kisin~\cite{Kis} (cf.~aussi \cite[cor.\,1.4.7]{Kis2})
\`a part pour le fait que Kisin ne fixe pas $M$ mais seulement
sa restriction au sous-groupe d'inertie et son d\'eterminant ce qui donne a priori deux
$M$ possibles: $M$ et $M\otimes\mu_{-1}$ o\`u $\mu_{-1}$ est le caract\`ere non ramifi\'e d'ordre~$2$.
(Les r\'esultats de Kisin sont valables pour des repr\'esentations
de~$\G_K$, avec $[K:\Q_p]<\infty$, en dimension arbitraire, mais comme notre construction
utilise la correspondance de Langlands locale $p$-adique, nous sommes forc\'es de nous
restreindre aux repr\'esentations de dimension~$2$ de $\G_{\Q_p}$;
la g\'eom\'etrie de
l'analogue de l'espace $X_{M,{\cal B}}^{[a,b]}$ ci-dessus a \'et\'e \'etudi\'ee par Wang-Erickson~\cite{WE}.)

(ii) Pour prouver que $R_{M,\overline\rho}^{[a,b]}$ est un produit d'anneaux principaux,
on commence par prouver que cet anneau est lisse et qu'il s'identifie \`a l'anneau des fonctions
analytiques born\'ees sur un ouvert de $\piqp$, ce qui permet d'utiliser 
la classification standard 
des ouverts de $\piqp$ (voir les travaux de
Rozensztajn~\cite{Roz} pour des r\'esultats du m\^eme type).

(iii)
Les deux ingr\'edients ci-dessus ont aussi \'et\'e utilis\'es par {\paskunas}~\cite{Pas2} pour prouver
la conjecture de Breuil-M\'ezard (avec quelques restrictions sur $\overline\rho$ dont certaines
ont \'et\'e lev\'ees par Hu et Tan~\cite{HT}), ce qui n\'ecessite
des informations sur $R_{M,\overline\rho}^{[a,b]}$ un peu diff\'erentes de celles mentionn\'ees ci-dessus
(et il ne faut pas rendre $p$ inversible), cf.~\cite[th.\,6.6, th.\,6.24]{Pas2}.

(iv) On d\'eduit de la construction de ${\rm LL}_{\cal B}^{[a,b]}(M)$ et 
des r\'esultats des \S\S\,\ref{PETIT9} et~\ref{PETIT10} une suite de d\'ecomposition 
de la r\'eduction de $\rho_M^{[a,b]}$ qui fournit (prop.\,\ref{BM1} et~\ref{BM3}) une forme
de la conjecture de Breuil-M\'ezard~\cite{BM,EG}. 

(v) L'\'enonc\'e du th.\,\ref{petit1} et sa preuve sont purement locaux, mais il y a quand-m\^eme un
ingr\'edient global cach\'e, \`a savoir la compatibilit\'e local-global d'Emerton~\cite{eternalpreprint}
qui permet de s'assurer que $M$ et ${\rm LL}(M)$ se correspondent bien par la correspondance
de Langlands locale classique.

(vi) La motivation premi\`ere pour les r\'esultats de cet article \'etait la d\'efinition
des objets apparaissant dans la factorisation de la cohomologie \'etale $p$-adique
de la tour de Drinfeld~\cite[th.\,0.1]{CDN5}.
\end{rema}

\Subsection*{Notations}
On note:

\quad $\bullet$ $G$ le groupe $\gl_2(\Q_p)$, 

\quad $\bullet$ $B$ le borel des matrice triangulaires sup\'erieures,

\quad $\bullet$ $Z$ le centre de $G$,

\quad $\bullet$ $K:=\gl_2(\Z_p)$ le sous-groupe compact maximal.

\vskip1mm
On note:

\quad $\bullet$  $\mu_\lambda$ le caract\`ere non ramifi\'e de $\Q_p^\dual$ prenant la valeur $\lambda$ en $p$
(dans un anneau o\`u $\lambda$ est inversible),
i.e. $\mu_\lambda(x)=\lambda^{v_p(x)}$.

\quad $\bullet$ $\epsilon$ le caract\`ere $x\mapsto x|x|$ de $\Q_p^\dual$.

\quad $\bullet$ $\chi$ le caract\`ere $\chi\circ\det$ de $G$,
si $\chi$ est un caract\`ere de $\Q_p^\dual$.

\subsubsection*{Remerciements}
Cet article fait \'echo \`a des travaux de Kisin, {\paskunas}, Emerton et Gee; nous les
remercions de toutes les belles id\'ees qu'ils ont introduites dans le sujet.
Nous remercions aussi le rapporteur pour sa lecture attentive et ses remarques. 

\section{La correspondance de Langlands locale $p$-adique}\label{PETIT6}
\Subsection{Repr\'esentations irr\'eductibles}
Si $0\leq r\leq p-1$ et si $\chi:\Q_p^\dual\to k_L^\dual$ est un caract\`ere continu, soit 
$W_{r,\chi}$ le $KZ$-module $({\rm Sym}^rk_L^2)\otimes\chi$ (o\`u $\matrice{p}{0}{0}{p}$ agit
trivialement sur $k_L^2$), et soit $I_{r,\chi}$
son induite compacte de $KZ$ \`a $G$. Soit $T_{\rm BL}$ l'op\'erateur de Barthel-Livn\'e.
Il r\'esulte des travaux de Barthel-Livn\'e~\cite{BL0,BL} et Breuil~\cite{Breuil} que
${\rm End}_{k_L[G]}I_{r,\chi}=k_L[T]$, o\`u $T$ agit par $T_{\rm BL}$, et que,
si $P\in k_L[T]$ est irr\'eductible,
alors $$\Pi_{r,\chi,P}:=I_{r,\chi}/P$$
est irr\'eductible sauf si $r=0$ ou $p-1$ et $P=T\pm 1$
(bien s\^ur, si $\deg P\geq 2$, alors $\Pi_{r,\chi,P}$ n'est pas absolument irr\'eductible).
\begin{rema}\label{class3}
Soit $B(\chi\epsilon^{r+1}\mu_T,\chi\mu_{T^{-1}})$ la $k_L[T,T^{-1}]$-repr\'esentation
$$B(\chi\epsilon^{r+1}\mu_T,\chi\mu_{T^{-1}})
:={\rm Ind}_B^G(\chi\mu_{T^{-1}}\otimes\chi\epsilon^{r}\mu_T)$$
On d\'eduit des r\'esultats de Barthel-Livn\'e que, si $P\in k_L[T]$ est irr\'eductible
et n'admet pas $0$ pour racine, alors
$$B_{r,\chi,P}:=(k_L[T,T^{-1}]/P)\otimes_{k_L[T,T^{-1}]}B(\chi\epsilon^{r+1}\mu_T,\chi\mu_{T^{-1}})$$
a m\^eme semi-simplifi\'ee
que $\Pi_{r,\chi,P}$ et lui est isomorphe sauf si $r=p-1$ et $P=T-\lambda$, avec $\lambda=\pm1$,
o\`u $\Pi_{r,\chi,P}$ est une extension de $\chi\mu_\lambda$ par ${\rm St}\otimes\chi\mu_\lambda$ 
tandis que $B_{r,\chi,P}$ est
une extension de  ${\rm St}\otimes\chi\mu_\lambda$
par $\chi\mu_\lambda$.
\end{rema}

Si $k$ est une extension de $k_L$, notons ${\rm Irr}_kG$ l'ensemble des
$k$-repr\'esentations de $G$, lisses, admissibles et {\it absolument irr\'eductibles}.
Berger~\cite{berger} a prouv\'e qu'un \'el\'ement de ${\rm Irr}_{\overline{\bf F}_p}G$
admet un caract\`ere central, et les r\'esultats de~\cite{BL0,BL,Breuil} fournissent
la description suivante de~${\rm Irr}_{k}G$ 

\begin{prop}\label{class1}
Les objets de ${\rm Irr}_{k_L}G$ sont:

$\bullet$ les $\chi$, pour $\chi:\Q_p^\dual\to k_L^\dual$ caract\`ere lisse,

$\bullet$ les ${\rm St}\otimes\chi$, pour $\chi:\Q_p^\dual\to k_L^\dual$ caract\`ere lisse,

$\bullet$ les $B(\chi_1,\chi_2):={\rm Ind}_B^G\,\chi_2\otimes\chi_1\epsilon^{-1}$,
pour $\chi_1,\chi_2:\Q_p^\dual\to k_L^\dual$ caract\`eres lisses, 
avec $\chi_1\neq\epsilon\chi_2$,

$\bullet$ les $\Pi_{r,\chi,T}$, pour $0\leq r\leq p-1$ et $\chi:\Q_p^\dual\to k^\dual$ caract\`ere lisse.
\end{prop}

\Subsection{Le foncteur $\Pi\mapsto{\bf V}(\Pi)$}
Soit ${\rm Tors}\,G$ la cat\'egorie des $\O_L[G]$-modules lisses, de longueur finie
(un tel module est tu\'e par $p^N$, si $N$ est sup\'erieur ou \'egal \`a la longueur).
Si $\delta:\Q_p^\dual\to \O_L^\dual$ est un caract\`ere continu, on note
${\rm Tors}^\delta\,G$ la sous-cat\'egorie des objets de caract\`ere central $\delta$.

Si $\Pi\in {\rm Tors}^\delta G$, on peut~\cite[th.\,IV.2.13, IV.2.14]{gl2}
associer \`a $\Pi$, de mani\`ere fonctorielle,
une repr\'esentation ${\bf V}(\Pi)$ de $\G_{\Q_p}$ {\og de type $\gl_2$\fg}
(cf.~th.\,\ref{pasku6} pour le sens \`a donner \`a cet \'enonc\'e: si on est en dimension~$2$,
l'op\'erateur $g+\delta\epsilon(g)g^{-1}$ qui y appara\^{\i}t est juste la trace de $g$).

Si $\Pi=\varprojlim \Pi_i$, o\`u les $\Pi_i$ sont de longueur finie, on pose
${\bf V}(\Pi)=\varprojlim{\bf V}(\Pi_i)$, et si $\Pi=L\otimes_{\O_L}\Pi^+$, on
pose ${\bf V}(\Pi)=L\otimes_{\O_L}{\bf V}(\Pi^+)$, ce qui permet de d\'efinir
${\bf V}(\Pi)$ pour des banachs dont la r\'eduction est de longueur finie (ou m\^eme
juste limite projective d'objets de longueur finie).

\begin{prop}
{\rm(\cite[th.\,0.10]{gl2} ou~\cite[th.\,1.2.4]{Kis2})}

{\rm (i)} ${\bf V}(\Pi)=0$ si $\Pi$ est de dimension~$1$.

{\rm (ii)} ${\bf V}(\Pi_{r,\chi,P})=\epsilon^{r+1}\mu_\lambda\chi$, si $P(\lambda)=0$ et $\lambda\neq 0$.

{\rm (iii)} ${\bf V}(\Pi_{r,\chi,T})={\rm Ind}_{\G_{\Q_{p^2}}}^{\G_{\Q_p}}\omega_2^{r+1}\otimes\chi$.
\end{prop} 

\Subsection{Blocs}\label{qc3}

La th\'eorie de Gabriel~\cite{gaby} fournit des d\'ecompositions
$${\rm Tors}\,G=\prod\nolimits_{\cal B}{\rm Tors}_{\cal B}\,G,\quad
{\rm Tors}^\delta\,G=\prod\nolimits_{\cal B}{\rm Tors}^\delta_{\cal B}\,G$$
o\`u ${\cal B}$ parcourt l'ensemble des 
blocs\footnote{On met, sur les classes d'isomorphisme de repr\'esentations irr\'eductibles,
 une relation d'\'equivalence d\'efinie par
$\pi\sim\pi'$ si et seulement si il existe une suite $\pi=\pi_0,\pi_1,\dots,\pi_r=\pi'$ telle que
$\pi_{i+1}=\pi_i$ ou ${\rm Ext}^1(\pi_i,\pi_{i+1})\neq 0$ ou ${\rm Ext}^1(\pi_{i+1},\pi_i)\neq 0$
(ces conditions ne sont pas exclusives).
Une classe d'\'equivalence est {\it un bloc}.} de ${\rm Tors}\,G$ (resp.~${\rm Tors}^\delta\,G$), et
${\rm Tors}_{\cal B}\,G$ (resp.~${\rm Tors}_{\cal B}^\delta\,G$) 
est la sous-cat\'egorie de ${\rm Tors}\, G$ (resp.~${\rm Tors}^\delta\,G$)
des objets dont toutes les composantes de Jordan-H\"older appartiennent \`a ${\cal B}$.
(Les blocs de ${\rm Tors}^\delta\,G$ sont ceux de ${\rm Tors}\,G$ de caract\`ere central $\delta$.)
Les composantes de Jordan-H\"older de $\Pi_{r,\chi,P}$ appartiennent
toutes \`a un m\^eme bloc. 

On dit qu'un bloc de ${\rm Tors}\,G$ ou ${\rm Tors}^\delta\,G$ est {\it absolu} s'il est
constitu\'e de repr\'esentations absolument irr\'eductibles (si ${\cal B}$ n'est pas absolu,
alors ${\cal B}$ peut \^etre consid\'er\'e comme
 absolu sur une extension non ramifi\'ee $L({\cal B})$ de $L$,
cf.~rem.\,\ref{absolu}).

\begin{prop}\label{pasku0}
{\rm (\paskunas)}
 Les blocs absolus de ${\rm Tors}\,G$ sont d'une des formes suivantes:

\quad $\bullet$ $\{\Pi_{r,\chi,T}\}$.

\quad $\bullet$ $\{B(\chi_1,\chi_2),B(\chi_2,\chi_1)\}$, o\`u
$\chi_1\chi_2^{-1}\neq \epsilon^{\pm}, 1$.

\quad $\bullet$ $\{B(\chi,\chi)\}$, si $p\neq 2$
{\rm(}devient $\{\chi, {\rm St}\otimes\chi\}$, si $p=2${\rm )}.

\quad $\bullet$ $\{\chi, {\rm St}\otimes\chi, B(\chi,\chi\epsilon)\}$, si $p\neq 3$
{\rm(}devient
$\{\chi, {\rm St}\otimes\chi, \chi\epsilon, {\rm St}\otimes\chi\epsilon\}$, si $p=3${\rm )}.
\end{prop}

Si $\pi$ est irr\'eductible, on note $P_\pi$ (resp.~$P_\pi^\delta$)
le dual de Pontryagin d'une enveloppe injective de 
$\pi$ dans la cat\'egorie des $\O_L[G]$-modules lisses qui sont la r\'eunion de leurs sous 
$\O_L[G]$-modules de longueur finie (resp. de caract\`ere central $\delta$). Le caract\`ere central de 
$P_\pi^\delta$ est donc $\delta^{-1}$, et 
 $P_\pi$ est un $\O_L[G]$-module compact, limite projective de $\O_L[G]$-modules compacts et
de longueur finie.

Si ${\cal B}$ est un bloc de ${\rm Tors}\,G$ ou de ${\rm Tors}^\delta\,G$, on pose
\begin{align*}
&P_{\cal B}=\oplus_{\pi\in {\cal B}} P_\pi
&&E_{\cal B}:={\rm End}_G(P_{\cal B})&& Z_{\cal B}:={\text{centre de $E_{\cal B}$}}\\
&P_{\cal B}^\delta=\oplus_{\pi\in {\cal B}} P_\pi^\delta
&&E_{\cal B}^\delta:={\rm End}_G(P_{\cal B}^\delta)
&& Z_{\cal B}^\delta:={\text{centre de $E_{\cal B}^\delta$}}
\end{align*}
Alors on dispose d'un morphisme naturel $Z_{\cal B}\to Z_{\cal B}^\delta$.

\begin{rema}\label{pasku1}
L'anneau $Z_{\cal B}^\delta$ s'identifie, 
par les r\'esultats de Gabriel \cite{gaby}, au centre de la cat\'egorie 
${\rm Tors}^\delta_{\cal B}\, G$; il agit donc naturellement sur tout 
 $\pi\in {\rm Tors}^\delta_{\cal B}\, G$ ainsi que, par fonctorialit\'e, sur ${\bf V}(\pi)$. 
L'action sur $\pi$ peut s'expliciter en utilisant la description 
 $$\pi=(P_{\cal B}\otimes_{E_{\cal B}}{\rm Hom}(P_{\cal B},\pi^\vee))^\vee$$
 o\`u $^\vee$ d\'esigne
le dual de Pontryagin.
\end{rema}

\Subsection{Blocs et repr\'esentations galoisiennes modulo~$p$}\label{qc10}
On associe \`a un bloc absolu ${\cal B}$ de ${\rm Tors}\, G$ la $k_L$-repr\'esentation 
$\rho_{\cal B}$ de $\G_{\Q_p}$ semi-simple, de dimension~$2$, dont les composantes
irr\'eductibles sont les ${\bf V}(\pi)$, pour $\pi\in{\cal B}$.
On obtient de la sorte une bijection entre les blocs
absolus de ${\rm Tors}\, G$ et les $k_L$-repr\'esentations   
$\rho_{\cal B}$ de $\G_{\Q_p}$ semi-simples, de dimension~$2$.
\begin{rema}\label{pasku3}
$\rho_{\cal B}$ v\'erifie les propri\'et\'es suivantes:

$\bullet$ $\rho_{\cal B}$ est irr\'eductible, de dimension~$2$, 
si ${\cal B}=\{\Pi_{r,\chi,T}\}$.

$\bullet$ $\rho_{\cal B}=\chi_1\oplus\chi_2$, 
si ${\cal B}=\{B(\chi_1,\chi_2),B(\chi_2,\chi_1)\}$ et $\chi_1\neq\chi_2$.

$\bullet$ $\rho_{\cal B}=\chi\oplus\chi$, si ${\cal B}=\{B(\chi,\chi)\}$. 

$\bullet$ $\rho_{\cal B}=\chi\oplus\chi\epsilon$, si ${\cal B}=\{\chi,{\rm St}\otimes\chi,B(\chi,\chi\epsilon)\}$.

Si ${\cal B}$ est un bloc de ${\rm Tors}^\delta_G$, alors $(\det\rho_{\cal B})\epsilon^{-1}=\delta$.
\end{rema}

Soit ${\cal B}$ un bloc absolu de ${\rm Tors}^\delta\,G$. Notons
$R_{\cal B}^{{\rm ps},\delta}$ l'anneau des d\'eformations universelles
du pseudo-caract\`ere ${\rm Tr}\circ\rho_{\cal B}$, de dimension~$2$ et
de d\'eterminant $\delta\epsilon$, et 
$$T_{\cal B}^\delta:\G_{\Q_p}\to R_{\cal B}^{{\rm ps},\delta}$$
 le pseudo-caract\`ere universel.
On a alors le {\og th\'eor\`eme $R=T$\fg} local suivant (cf.~rem.\,\ref{pasku1} pour l'action
de $Z_{\cal B}^\delta$).

\begin{theo}\label{pasku6}
{\rm (\paskunas~\cite{Paskext}, {\paskunas}-Tung~\cite{PT})}
Il existe un unique morphisme d'anneaux
$$\iota_{\cal B}^\delta:R_{\cal B}^{{\rm ps},\delta}\to Z_{\cal B}^\delta$$
 tel que, pour tout $\pi\in {\rm Tors}^\delta_{\cal B} G$ et tout $g\in{\rm Gal}_{\Q_p}$,
on ait:
$$\iota_{\cal B}^\delta(T_{\cal B}^\delta(g))=g+\delta\epsilon(g)g^{-1}
\quad{\text{dans ${\rm End}({\bf V}(\pi))$}}.$$
De plus:

$\bullet$ Si $p\geq 5$, $\iota_{\cal B}^\delta$ est un isomorphisme.

$\bullet$ Dans le cas g\'en\'eral $\iota_{\cal B}^\delta$ est un morphisme fini 
{\rm (i.e. $Z_{\cal B}^\delta$ est un 
$R_{\cal B}^{{\rm ps}}$-module de type fini)}, $L\otimes_{\O_L}\iota_{\cal B}^\delta$ est un isomorphisme,
$Z_{\cal B}^\delta=R_{\cal B}^{{\rm ps},\delta}/\O_L{\text{-{\rm torsion}}}$ si $p=3$ et le conoyau
de $R_{\cal B}^{{\rm ps},\delta}/\O_L{\text{-{\rm torsion}}}\to Z_{\cal B}^\delta$ est tu\'e par $2$ si $p=2$.

$\bullet$ Le $Z_{\cal B}^\delta$-module $E_{\cal B}$ est de type fini.
\end{theo}
\begin{rema}
Les anneaux de d\'eformations universelles de pseudo-caract\`eres de $\G_{\Q_p}$ sont
noeth\'eriens (et m\^emes quotients de $\O_L[[x_1,\dots,x_r]]$); on en d\'eduit que
$Z_{\cal B}^\delta[1/p]$ est noeth\'erien. En fait $Z_{\cal B}^\delta$ est noeth\'erien \cite[th 1.1]{PT}. 
\end{rema}

\begin{rema}\label{absolu}
Tout ce qui pr\'ec\`ede suppose que ${\cal B}$ est absolu mais 
on peut consid\'erer tout bloc comme absolu, quitte \`a remplacer $L$ par une
certaine extension non ramifi\'ee $L({\cal B})$.  En effet:

{\rm (i)} Si $\pi$ est irr\'eductible comme $k_L[G]$-module,
 alors l'anneau $k(\pi):={\rm End}_{k[G]}\pi$ est une extension finie
de $k$, et $\pi$ est absolument irr\'eductible vue comme $k(\pi)[G]$-module.

{\rm (ii)} Si ${\cal B}$ est un bloc de ${\rm Tors}\,G$, 
alors $k(\pi)$ ne d\'epend pas de $\pi\in{\cal B}$; notons le
$k({\cal B})$, et notons $L({\cal B})$ l'extension non ramifi\'ee de $L$ de corps
r\'esiduel $k({\cal B})$. 

{\rm (iii)} Tout \'el\'ement de ${\rm Tors}_{\cal B}G$
est muni d'une action de $\O_{L({\cal B})}$ relevant celle de $k({\cal B)}$,
ce qui permet de consid\'erer ${\cal B}$ comme un bloc absolu de ${\rm Tors}_{\O_{L({\cal B})}}G$.
De plus, l'oubli de l'action de $\O_{L({\cal B})}$ induit une \'equivalence entre
${\rm Tors}_{\O_{L({\cal B})},{\cal B}}G$ et ${\rm Tors}_{\cal B}G$.

{\rm (iv)} Si $\Pi\in {\rm Tors}_{\cal B}G$, alors
$\O_{L({\cal B})}\otimes_{\O_L}\Pi$ se decompose sous la forme $\oplus_{\sigma}\Pi^\sigma$,
o\`u $\sigma$ d\'ecrit ${\rm Hom}_L(L({\cal B}),L({\cal B}))$. 
Cela d\'ecoupe $\O_{L({\cal B})}\otimes_{\O_L}{\cal B}$
en blocs absolus ${\cal B}^\sigma$ qui sont distincts mais conjugu\'es sous l'action
de ${\rm Gal}(k({\cal B})/k_L)$.
\end{rema}

\section{Compl\'etions profinie et ${\cal B}$-adique}\label{PETIT7}
\Subsection{Compl\'etion profinie}\label{PETIT8}
Si $X$ est un $\O_L[G]$-module topologique, on note $\widetilde X$ son {\it compl\'et\'e profini},
i.e.~la limite projective des $X/W$,
o\`u $W$ parcourt l'ensemble des sous-$\O_L[G]$-modules
ouverts (et donc aussi ferm\'es) de $X$ tels que $X/W\in {\rm Tors}\,G$.
Alors $\widetilde{X}$ admet une d\'ecomposition
$$\widetilde{X}=\prod\nolimits_{\cal B}{X}_{\cal B}$$
o\`u ${\cal B}$ d\'ecrit
l'ensemble des blocs des repr\'esentations de $G$ sur $k_L$;
${X}_{\cal B}$
est le {\it compl\'et\'e ${\cal B}$-adique} de $X$.
Si $X$ a pour caract\`ere central $\delta$, les blocs apparaissant dans le produit sont ceux
de ${\rm Tors}^\delta\,G$.
\begin{rema}\label{petit5}
On n'est pas forc\'e de prendre un bloc pour d\'efinir un compl\'et\'e ${\cal B}$-adique: 
on peut prendre pour ${\cal B}$ un sous-ensemble d'un bloc et ne consid\'erer
que les quotients dont les composantes de Jordan-H\"older appartiennent \`a ${\cal B}$. 
\end{rema}

\begin{prop}\label{GAB1}
Si $X$ est une repr\'esentation lisse de $G$, de type fini, \`a caract\`ere central $\delta$, on dispose d'un isomorphisme naturel de $G$-modules ind-profinis 
 $$X_{\cal B}^{\vee}\simeq P^{\delta}_{\cal B}\otimes_{E^{\delta}_{\cal B}} {\rm Hom}_G(P^{\delta}_{\cal B}, X^{\vee}).$$
\end{prop}

\begin{proof} Soit $(\sigma_i)_{i\in I}$ l'ensemble (d\'enombrable) des quotients de $X$ qui appartiennent \`a ${\rm Tors}^{\delta}_{\cal B}$. On a donc $X_{\cal B}^{\vee}\simeq \varinjlim_{i\in I} \sigma_i^{\vee}$ et pour tout $i$ la fl\`eche naturelle 
$P^{\delta}_{\cal B}\otimes_{E^{\delta}_{\cal B}} {\rm Hom}_G(P^{\delta}_{\cal B}, \sigma_i^{\vee})\to \sigma_i^{\vee}$ est un isomorphisme, d'o\`u un isomorphisme
$$X_{\cal B}^{\vee}\simeq P^{\delta}_{\cal B}\otimes_{E^{\delta}_{\cal B}} (\varinjlim_{i\in I} {\rm Hom}_G(P^{\delta}_{\cal B}, \sigma_i^{\vee}))).$$
 Il suffit donc de montrer que l'injection $\varinjlim_{i\in I} {\rm Hom}_G(P^{\delta}_{\cal B}, \sigma_i^{\vee})
\to {\rm Hom}_G(P^{\delta}_{\cal B}, X^{\vee})$ (induite par 
 les injections $\sigma_i^{\vee}\to X^{\vee}$) est une surjection. Soit donc $f: P^{\delta}_{\cal B}\to X^{\vee}$ un morphisme continu
 $G$-\'equivariant. Son image est ferm\'ee car $P^{\delta}_{\cal B}$ est compact,
et elle est $G$-stable dans $X^{\vee}$, donc son dual de Pontryagin correspond \`a un quotient $\sigma$ de $X$. Il suffit de montrer que $\sigma$ est de longueur finie. Mais $\sigma^{\vee}$ est un quotient de $P^{\delta}_{\cal B}$, donc $\sigma$ est un sous-objet de 
 $(P^{\delta}_{\cal B})^{\vee}$, et ce dernier est une limite inductive de repr\'esentations de longueur finie. D'autre part $\sigma$ est un quotient de $X$, donc 
 $\sigma$ est de type fini, et on conclut que $\sigma$ est de longueur finie. 
\end{proof}

\begin{coro}\label{GAB2}
Soit $\mathcal{C}^{\delta}$ la cat\'egorie des 
 repr\'esentations lisses de type fini de $G$, \`a caract\`ere central $\delta$. 
 Le foncteur $X\mapsto X_{\cal B}$ de $\mathcal{C}^{\delta}$ dans la cat\'egorie des $G$-modules pro-discrets est exact. 
\end{coro}
\begin{proof}
  Soit $0\to X_1\to X\to X_2\to 0$ une suite exacte dans $\mathcal{C}^{\delta}$. Le 
  $G$-module lisse $(P^{\delta}_{\cal B})^{\vee}$ est injectif dans la cat\'egorie de toutes les repr\'esentations lisses de 
  $G$, \`a caract\`ere central 
  $\delta$ (prop. $5.16$ de \cite{Paskext}). On en d\'eduit que 
  la suite $$0\to {\rm Hom}_{G}(P^{\delta}_{\cal B}, X_2^{\vee})\to  {\rm Hom}_{G}(P^{\delta}_{\cal B}, X^{\vee})\to  {\rm Hom}_{G}(P^{\delta}_{\cal B}, X_1^{\vee})\to 0$$ reste exacte. Pour conclure via la proposition \ref{GAB1} il suffit de voir que 
$P^{\delta}_{\cal B}$ est plat sur $E^{\delta}_{\cal B}$. Mais la th\'eorie de Gabriel \cite{gaby} (voir aussi le chapitre $2$ de 
\cite{Paskext}) 
montre que les foncteurs ${\rm Hom}_{G}(P^{\delta}_{\cal B}, (-)^{\vee})$
et  $(-)\wotimes_{E^{\delta}_{\cal B}} P^{\delta}_{\cal B}$ induisent une \'equivalence de cat\'egories entre 
celle des repr\'esentations localement admissibles de $G$, \`a caract\`ere central $\delta$ et 
dont les sous-quotients irr\'eductibles sont dans $\cal B$, et celle des 
$E^{\delta}_{\cal B}$-modules compacts. Pour v\'erifier la platitude de 
$P^{\delta}_{\cal B}$ sur $E^{\delta}_{\cal B}$ il suffit de tester l'exactitude du foncteur $(-)\otimes_{E^{\delta}_{\cal B}} P^{\delta}_{\cal B}$ sur les $E^{\delta}_{\cal B}$-modules de type fini. 
Si $M$ est un tel module on a 
$M\otimes_{E^{\delta}_{\cal B}} P^{\delta}_{\cal B}\simeq M\wotimes_{E^{\delta}_{\cal B}} P^{\delta}_{\cal B}$. Comme le foncteur 
${\rm Hom}_{G}(P^{\delta}_{\cal B}, (-)^{\vee})$ est exact (cf. le d\'ebut de la preuve), il en est de m\^eme de son quasi-inverse 
$(-)\wotimes_{E^{\delta}_{\cal B}} P^{\delta}_{\cal B}$, ce qui permet de conclure. 
\end{proof}

\begin{coro}\label{fini}
 Si $X$ est une repr\'esentation lisse de type fini de $G$, \`a caract\`ere central $\delta$, alors ${\bf V}(X_{\cal B})$ est de type fini sur $R_{\cal B}^{{\rm ps}, \delta}$.
\end{coro}

\begin{proof} 
Le $R_{\cal B}^{{\rm ps}, \delta}$-module ${\bf V}(X_{\cal B})$ est une limite inverse de modules de la forme ${\bf V}(\pi)$ avec $\pi$ de longueur finie. Chaque 
${\bf V}(\pi)$ est de longueur finie sur $R_{\cal B}^{{\rm ps}, \delta}$, donc ${\bf V}(X_{\cal B})$ est un $R_{\cal B}^{{\rm ps}, \delta}$-module compact. 
Soit $m$ l'id\'eal maximal de $R_{\cal B}^{{\rm ps}, \delta}$. Par le lemme de Nakayama topologique il suffit de montrer que ${\bf V}(X_{\cal B})/m {\bf V}(X_{\cal B})\simeq {\bf V}(X_{\cal B}/m X_{\cal B})$ est 
de dimension finie sur $k_L$, et pour cela 
il suffit de voir que 
$X_{\cal B}/m X_{\cal B}$ est de longueur finie comme $k_L[G]$-module, ou encore que $X_{\cal B}^{\vee}[m]$ est de longueur finie comme $k_L[G]$-module compact. 
Comme $P_{\cal B}^{\delta}$ est plat sur $E_{\cal B}^{\delta}$, par la proposition ci-dessus et sa preuve, il suffit de v\'erifier que $P^{\delta}_{\cal B}/m P^{\delta}_{\cal B}\otimes_{E^{\delta}_{\cal B}/m E^{\delta}_{\cal B}} {\rm Hom}_G(P^{\delta}_{\cal B}, X^{\vee})[m]$ est de longueur finie.
Puisque $(P_{\cal B}/m P_{\cal B})^{\vee}$ est de longueur finie (prop. $6.7$ de \cite{PT}), 
il suffit de voir que ${\rm Hom}_G(P_{\cal B}, X^{\vee})[m]$ est de dimension finie sur $k_L$.
 Cela d\'ecoule du fait que $(P_{\cal B}/m P_{\cal B})^{\vee}$ est admissible (car de longueur finie) et que ${\rm Hom}_G(X, \pi)$ est de dimension finie quand 
$X$ est de type fini et $\pi$ est admissible.
\end{proof}

\Subsection{Compl\'etion ${\cal B}$-adique des induites compactes: le cas irr\'eductible}\label{PETIT9}
\begin{prop}\label{A33}
Si ${\cal B}$ est le bloc contenant
 $\Pi_{r,\chi,P}$, alors
$${(I_{r,\chi})}_{\cal B}=
{\varprojlim}_n I_{r,\chi}/P^n$$
De plus,
$${\rm End}_{k_L[G]}({(I_{r,\chi})}_{\cal B})=
{\varprojlim}_n k_L[T]/P^n $$
\end{prop}
\begin{proof}
Tout sous-$G$-module non nul de $I_{r,\chi}$ contient $Q\cdot I_{r,\chi}$
avec $Q\neq 0$ (ceci est standard, voir par exemple le cor. $2.1.4$ de \cite{DGE}), et donc
 un quotient $\Pi$
de $I_{r,\chi}$, de longueur finie, est quotient de $I_{r,\chi}/Q$, avec $Q\in k_L[T]$, non nul.
Le th\'eor\`eme des restes chinois implique donc que $\Pi$ est quotient de
$\oplus_{P}I_{r,\chi}/P^{n_P}$, o\`u la somme porte
sur les $P$ irr\'eductibles et les $n_P$ sont presque tous nuls.
Le premier point s'en d\'eduit en remarquant que les blocs associ\'es aux diff\'erents $P$
sont distincts (\`a $r,\chi$ fix\'e, s'entend).

Le second d\'ecoule de ce que ${\rm End}(I_{r,\chi}/Q)=k_L[T]/Q$. En effet, 
la suite exacte $0\to I_{r,\chi}\to I_{r,\chi}\to I_{r,\chi}/Q\to 0$ et l'isomorphisme
${\rm End}( I_{r,\chi})\simeq k_L[T]$ fournissent une injection de $k_L[T]/Q$ dans ${\rm End}(I_{r,\chi}/Q)$, il suffit donc de v\'erifier que 
ce dernier $k_L$-espace vectoriel est de dimension $\leq \deg(Q)$. Pour cela on peut \'etendre les scalaires \`a une cl\^oture alg\'ebrique de $k_L$ (en utilisant le lemme $5.1$ de \cite{Paskext}), 
et on se ram\`ene par d\'evissage au cas $\deg Q=1$, qui est standard. 
\end{proof}

\begin{prop}
Si $P\in k_L[T]$ est irr\'eductible
et n'admet pas $0$ pour racine, si ${\cal B}$ est le bloc qui lui correspond,
et si $T$ agit sur $(I_{r,\chi})_{\cal B}$ via $T_{\rm BL}$,
on a un isomorphisme de $k_L[T]_P[G]$-modules\footnote{$k_L[T]_P$ est le compl\'et\'e $P$-adique de $k_L[T]$.}:
$$(I_{r,\chi})_{\cal B}\cong k_L[T]_P\wotimes_{k_L[T,T^{-1}]}B(\chi\epsilon^{r+1}\mu_T,\chi\mu_{T^{-1}})$$
dans le cas g\'en\'erique, et une suite exacte
$$0\to k_L[T]_P\wotimes_{k_L[T,T^{-1}]}B(\chi\epsilon^{r+1}\mu_T,\chi\mu_{T^{-1}})\to
(I_{r,\chi})_{\cal B}\to \chi\mu_\lambda\to 0$$
si $r=p-1$ et $P=T-\lambda$ avec $\lambda=\pm1$.
\end{prop}
\begin{proof}
On a une fl\`eche naturelle $I_{r,\chi}\to B(\chi\epsilon^{r+1}\mu_T,\chi\mu_{T^{-1}})$
de $k_L[T][G]$-modules, o\`u $T$ agit sur $I_{r,\chi}$ par $T_{\rm BL}$.
Celle-ci induit une fl\`eche de $k_L[T,T^{-1}][G]$-modules
$$k_L[T,T^{-1}]\otimes_{k_L[T]}I_{r,\chi}\to B(\chi\epsilon^{r+1}\mu_T,\chi\mu_{T^{-1}})$$
qui est injective et est un isomorphisme sauf si $r=p-1$ o\`u l'on a une suite exacte
$$0\to k_L[T,T^{-1}]\otimes_{k_L[T]}I_{r,\chi}\to B(\chi\epsilon^{r+1}\mu_T,\chi\mu_{T^{-1}})
\to ({\rm St}\otimes\chi\mu_1)
\oplus ({\rm St}\otimes\chi\mu_{-1})\to 0$$
et $T$ agit par $1$ (resp. $-1$) sur le premier (resp.~second) terme du conoyau.

$\bullet$
Si $r\neq p-1$ ou si $P\neq T-\lambda$, avec $\lambda\neq\pm1$ (cas g\'en\'erique), 
on a donc, pour tout~$n$, un isomorphisme
$$I_{r,\chi}/P^n\cong (k_L[T]/P^n)\otimes_{k_L[T,T^{-1}]}
B(\chi\epsilon^{r+1}\mu_T,\chi\mu_{T^{-1}})$$

$\bullet$ 
Si $r= p-1$ et si $\lambda=\pm1$ (cas exceptionnels), on a, pour tout $n$, des suites exactes
\begin{align*}
0\to {\rm St}\otimes \chi\mu_\lambda\to 
I_{r,\chi}/(T-\lambda)^n\to &M_{\lambda,n}\to 0\\
0\to (T-\lambda)(B(\chi\epsilon^{r+1}\mu_T,\chi\mu_{T^{-1}})/(T-\lambda)^{n-1})\to
&M_{\lambda,n} \to \chi\mu_\lambda\to0\\
M_{\lambda,n}:={\rm Ker}
\big(B(\chi\epsilon^{r+1}\mu_T,\chi\mu_{T^{-1}})/(T-\lambda)^n
\to {\rm St}\otimes \chi\mu_\lambda\big)&
\end{align*}
Quand on passe de $n+1$ \`a $n$, 
la fl\`eche ${\rm St}\otimes\chi\mu_\lambda\to {\rm St}\otimes\chi\mu_\lambda$ est $0$ (et pas l'identit\'e).
En passant aux limites projectives, cela fournit, dans le cas g\'en\'erique,
un isomorphisme
$$(I_{r,\chi})_{\cal B}\cong 
k_L[T]_P\wotimes_{k_L[T,T^{-1}]}
B(\chi\epsilon^{r+1}\mu_T,\chi\mu_{T^{-1}})$$
et dans les cas exceptionnels, un isomorphisme
$(I_{r,\chi})_{\cal B}\cong\varprojlim_nM_{\lambda,n}$, et une suite exacte
$$0\to P\, k_L[T]_P\wotimes_{k_L[T,T^{-1}]}
B(\chi\epsilon^{r+1}\mu_T,\chi\mu_{T^{-1}})\to (I_{r,\chi})_{\cal B} \to \chi\mu_\lambda\to 0$$
On en d\'eduit le r\'esultat.
\end{proof}

\begin{lemm}\label{petit6}
{\rm(\cite[Lemma\,1.5.11]{Kis2})}
Si $P\in k_L[T]$ est irr\'eductible et si ${\cal B}$ est le bloc que lui est associ\'e,
l'application naturelle
$\alpha_{r,\chi}:Z_{\cal B}\to {\rm End}((I_{r,\chi})_{\cal B})$ est surjective
sauf dans le cas $r=p-2$ et $P=T\pm 1$, correspondant \`a ${\cal B}=\{B(\chi\mu_{\pm1},\chi\mu_{\pm1})\}$,
o\`u l'image est $k_L[[P^2]]$.
\end{lemm}

\begin{coro}\label{A34}
{\rm(\cite[Lemma\,1.5.2]{Kis2})}
${\bf V}({(I_{r,\chi})}_{\cal B})$ est un module de type fini sur~$Z_{\cal B}$, de rang~$2$, 
si $P=T$ et {\og de rang $1$\fg} si $P$ n'a pas $0$ comme racine.
\end{coro}
\begin{proof}
Cela suit de ce que ${\bf V}({(I_{r,\chi})}_{\cal B}/P)$ est de rang fini
 sur ${\rm End}({(I_{r,\chi})}_{\cal B})/P$, et ce rang vaut $2$ si $P=T$ (car ${(I_{r,\chi})}_{\cal B}/P$
est supersinguli\`ere) et $1$ sinon (car ${(I_{r,\chi})}_{\cal B}/P$ est une s\'erie principale 
\`a des caract\`eres pr\`es).
\end{proof}
\begin{rema}\label{A35}
Dans le cas d'un bloc de la s\'erie principale,
${\bf V}({(I_{r,\chi})}_{\cal B})$
est la d\'eformation non ramifi\'ee d'un caract\`ere. 
Dans le cas ${\cal B}$ supersingulier, on obtient une d\'eformation {\og Fontaine-Laffaille\fg}
de ${\bf V}(\Pi_{r,\chi,T})$, cf.~\cite[Lemma\,1.5.3]{Kis2}. 
\end{rema}

\Subsection{Compl\'etion ${\cal B}$-adique des induites compactes: exactitude}\label{PETIT10}
On fixe $\delta:\Q_p^\dual\to\O_L^\dual$ un caract\`ere lisse, et tous nos $KZ$-modules et $G$-modules
sont suppos\'es de caract\`ere central $\delta$.
Si $W$ est une repr\'esentation de $KZ$, de type fini sur $\O_L$, on pose
$$I(W):={\text{c-Ind}}_{KZ}^G\,W$$
Comme $W\mapsto I(W)$ est exact, on d\'eduit du cor.\,\ref{GAB2} le r\'esultat suivant:
\begin{coro}\label{GAB3}
Les foncteurs $W\mapsto \widetilde{I(W)}$ et $W\mapsto I(W)_{\cal B}$
sont exacts.
\end{coro}

\begin{coro}\label{petit11}
Soit $W$ un $\O_L[KZ]$-module sans $p$-torsion.

{\rm (i)} $\widetilde{I(W)}$ est sans $p$-torsion.

{\rm (ii)} $\widetilde{I(W)}=\varprojlim_n\widetilde{I(W)}/p^n$,
et $\widetilde{I(W)}/p^n=\widetilde{I(W/p^n)}$, pour tout $n$.
\end{coro}
\begin{proof}
L'identification $\widetilde{I(W)}=\varprojlim_n\widetilde{I(W/p^n)}$ est imm\'ediate.
La multiplication par $p^n$ induit un isomorphisme $I(W)/p^n\overset{\sim}{\to}p^nI(W)/p^{2n}I(W)\subset
I(W)/p^{2n}$; elle induit donc aussi un isomorphisme
$\widetilde{I(W)}/p^n\overset{\sim}{\to}p^n\widetilde{I(W)}/p^{2n}\widetilde{I(W)}\subset
\widetilde{I(W)}/p^{2n}$. A la limite, on voit qu'un \'el\'ement de $p$-torsion
de $\varprojlim_n\widetilde{I(W)}/p^n$ est dans l'image de la multiplication par $p^n$, pour tout $n$.
Le r\'esultat s'en d\'eduit car le sous-groupe $p$-divisible de $\widetilde{I(W)}$ est r\'eduit
\`a $0$ (vu que c'est le cas pour tous les quotients de longueur finie de $I(W)$).
\end{proof}

\begin{rema}\label{petit12}
En utilisant la d\'ecomposition suivant les blocs, on en d\'eduit
les m\^emes r\'esultats pour la compl\'etion ${\cal B}$-adique.
\end{rema}

\section{Le compl\'et\'e universel de ${\rm LL}^{[a,b]}(M)$}\label{PETIT5}
Soit $M$ un $(\varphi,N,\G_{\Q_p})$-module de rang $2$ sur $L\otimes\Q_p^{\rm nr}$,
absolument irr\'eductible en restriction au sous-groupe d'inertie de $\G_{\Q_p}$ (et donc $N=0$),
et soit ${\rm LL}(M)$ la $L$-repr\'esentation lisse de $G$ associ\'ee par la correspondance
de Langlands locale classique (elle est supercuspidale gr\^ace \`a l'hypoth\`ese faite sur $M$).

Si $a<b$ sont des entiers, notons ${\rm LL}^{[a,b]}(M)$ la repr\'esentation localement
alg\'ebrique 
$${\rm LL}^{[a,b]}(M):={\rm LL}(M)\otimes {\rm Sym}^{b-a-1}\otimes{\det}^a$$
On suppose 
que le caract\`ere central $\delta_M^{[a,b]}$ de ${\rm LL}^{[a,b]}(M)$ est unitaire.

D'apr\`es la classification des repr\'esentations supercuspidales
de $G$, il existe une repr\'esentation irr\'eductible $\sigma_M$ de $KZ$ (de dimension finie),
telle que\footnote{On note ${\rm ind}:={\text{c-Ind}}$ l'induite \`a support compact.}
$${\rm ind}_{KZ}^G\sigma_M={\rm LL}(M)\quad{\text{ou bien}}\quad 
{\rm ind}_{KZ}^G\sigma_M={\rm LL}(M)\oplus ({\rm LL}(M)\otimes\mu_{-1})$$
On a donc le m\^eme r\'esultat pour ${\rm LL}^{[a,b]}(M)$ en rempla\c{c}ant
$\sigma_M$ par 
$$\sigma_M^{[a,b]}:=\sigma_M\otimes {\rm Sym}^{b-a-1}\otimes{\det}^a$$
c'est-\`a-dire:
$${\rm ind}_{KZ}^G\sigma_M^{[a,b]}={\rm LL}^{[a,b]}(M)\quad{\text{ou bien}}\quad
{\rm ind}_{KZ}^G\sigma_M^{[a,b]}={\rm LL}^{[a,b]}(M)\oplus ({\rm LL}^{[a,b]}(M)\otimes\mu_{-1})$$

L'hypoth\`ese selon laquelle le caract\`ere central de ${\rm LL}^{[a,b]}(M)$ est unitaire
implique qu'il existe des r\'eseaux de $\sigma_M^{[a,b]}$ stables par $KZ$.
Si $\sigma^+$ est un tel r\'eseau, on note $I_M(\sigma^+)$ l'image de
${\rm ind}_{KZ}^G\sigma^+$ dans ${\rm LL}^{[a,b]}(M)$; c'est un r\'eseau de ${\rm LL}^{[a,b]}(M)$
et $\widehat{ I_M(\sigma^+)}=\varprojlim I_M(\sigma^+)/\varpi^n$ est un r\'eseau
du compl\'et\'e universel
 $\widehat{\rm LL}\hskip0mm^{[a,b]}(M)$ de ${\rm LL}^{[a,b]}(M)$.

\begin{prop}\label{A32}
${\rm End}_{L[G]}(\widehat{\rm LL}\hskip0mm^{[a,b]}(M))=L$.
\end{prop}
\begin{proof}
Comme ${\rm LL}^{[a,b]}(M)$ est dense dans $\widehat{\rm LL}\hskip0mm^{[a,b]}(M)$
et comme
$${\rm End}_{L[G]}({\rm LL}^{[a,b]}(M))=L$$ 
d'apr\`es le lemme de Schur classique, il
suffit de prouver que les vecteurs localement alg\'ebriques de $\widehat{\rm LL}\hskip0mm^{[a,b]}(M)$ sont r\'eduits
\`a ${\rm LL}^{[a,b]}(M)$, et pour cela il suffit de prouver le m\^eme r\'esultat
pour ${\rm ind}_{KZ}^G\sigma_M^{[a,b]}$. 

Notons $X_n$ la double classe $KZ\matrice{p^n}{0}{0}{1}KZ$
(cela correspond aux points \`a distance~$n$ du sommet central
sur l'arbre de ${\rm PGL}_2(\Q_p)$). Alors ${\rm ind}_{KZ}^G\sigma_M^{[a,b]}=\oplus_n {\rm ind}_{KZ}^{X_n}\sigma_M^{[a,b]}$,
et le compl\'et\'e universel de ${\rm ind}_{KZ}^G\sigma_M^{[a,b]}$ est l'ensemble des
$x=\sum_{n\geq 0}x_n$, avec $x_n\in {\rm ind}_{KZ}^{X_n}\sigma_M^{[a,b]}$, et $x_n\to 0$ quand $n\to\infty$
(i.e. $x_n\in p^{k_n}{\rm ind}_{KZ}^{X_n}\sigma^+$, avec $k_n\to +\infty$).
L'application $x\mapsto R_n(x)=\sum_{i\leq n}x_i$ est $K$-\'equivariante et $x=\lim_{n\to\infty}R_n(x)$.
Si $x$ est $K_r$-alg\'ebrique, o\`u $K_r=1+p^r{\rm M}_2(\Z_p)$, il en est de m\^eme de $R_n(x)$ pour tout $n$.
Comme ${\rm ind}_{KZ}^G\sigma_M^{[a,b]}$ est admissible, les vecteurs $K_r$-alg\'ebriques dans
$\oplus_{i\leq n}{\rm ind}_{KZ}^{X_i}\sigma_M^{[a,b]}$ ne d\'ependent pas de $n\geq n(r)$; 
on en d\'eduit que $R_n(x)=R_{n(r)}(x)$
pour tout $n\geq n(r)$, et donc que $x=R_{n(r)}(x)\in {\rm ind}_{KZ}^G\sigma_M^{[a,b]}$.
En r\'esum\'e, les vecteurs localement alg\'ebriques du compl\'et\'e universel de
${\rm ind}_{KZ}^G\sigma_M^{[a,b]}$ sont r\'eduits \`a ${\rm ind}_{KZ}^G\sigma_M^{[a,b]}$.

Le r\'esultat s'en d\'eduit.
\end{proof}

Soit $\Pi$ un $L[G]$-banach unitaire de longueur finie, contenant
${\rm LL}^{[a,b]}(M)$ comme sous-espace dense.
Par propri\'et\'e universelle de $\widehat{\rm LL}\hskip0mm^{[a,b]}(M)$ l'injection
${\rm LL}^{[a,b]}(M)\hookrightarrow \Pi$ se prolonge en une application continue
$\widehat{\rm LL}\hskip0mm^{[a,b]}(M)\to \Pi$.
Comme $\widehat{\rm LL}\hskip0mm^{[a,b]}(M)$ n'est pas admissible, il n'y a aucune raison
{\it a priori}\footnote{Par exemple, si $G=\Z_p$, l'inclusion ${\cal C}^1(\Z_p)\hookrightarrow {\cal C}(\Z_p)$
est d'image dense mais n'est pas bijective (et ${\cal C}(\Z_p)$ est admissible
comme repr\'esentation de $\Z_p$, pas ${\cal C}^1(\Z_p)$).}
 pour que l'image soit ferm\'ee, mais on a le r\'esultat suivant.
\begin{prop}\label{A12}
L'application $\widehat{\rm LL}\hskip0mm^{[a,b]}(M)\to \Pi$ est surjective, et donc $\Pi$ est un quotient
de $\widehat{\rm LL}\hskip0mm^{[a,b]}(M)$.
\end{prop}
\begin{proof}
La surjectivit\'e de $\widehat{\rm LL}\hskip0mm^{[a,b]}(M)\to\Pi$ est une cons\'equence
du fait que $\Pi$ est r\'esiduellement de longueur finie~\cite{Paskext,CDP}: on fixe un r\'eseau
$\Pi^+$ de $\Pi$, et on note $r$ la longueur
de $\Pi^+/\varpi$.  Comme ${\rm LL}^{[a,b]}(M)$ est dense dans $\Pi$,
on peut trouver $v_1,\dots,v_r\in {\rm LL}^{[a,b]}(M)\cap\Pi^+$
dont les images modulo~$\varpi$ engendrent $\Pi^+/\varpi$.
Mais alors le sous-$\O_L[G]$-module $W$ de ${\rm LL}^{[a,b]}(M)$ engendr\'e par $v_1,\dots,v_r$
est de type fini, et donc son compl\'et\'e $\varpi$-adique $\widehat W$ peut \^etre
choisi comme boule unit\'e de $\widehat{\rm LL}\hskip0mm^{[a,b]}(M)$.
Par construction $\widehat{\rm LL}\hskip0mm^{[a,b]}(M)\to\Pi$ envoie $\widehat W$
dans $\Pi^+$ et induit une surjection modulo~$\varpi$.
On en d\'eduit, puisque tout est complet pour la topologie $\varpi$-adique,
que $\widehat W\to \Pi^+$ est surjective, ce qui permet de conclure.
\end{proof}

\section{Repr\'esentations de type $M$}
On va associer un certain nombre d'objets
\`a un $(\varphi,N,\G_{\Q_p})$-module $M$ satisfaisant les hypoth\`eses
du chap.\,\ref{PETIT5} (en plus de la repr\'esentation ${\rm LL}^{[a,b]}(M)$ d\'ej\`a introduite). 
\Subsection{D\'eformations de Rham et vecteurs localement alg\'ebriques}\label{defo1}
Soit $V_0$ une $L$-repr\'esentation de dimension~$2$ de $\G_{\Q_p}$, absolument irr\'eductible.
Soient $\delta=\det V_0$, et ${\cal B}$ le bloc correspondant \`a la r\'eduction de $V_0$.
Alors $V_0$ correspond \`a un point $x\in X:={\rm Spec}\,R^{{\rm ps},\delta}_{\cal B}[\frac{1}{p}]$,
et la restriction de $T_{\cal B}^\delta$ \`a une boule ouverte~$B$ de l'espace
rigide associ\'e \`a $X$, contenant $x$ et suffisamment petite, est la trace\footnote{
D'apr\`es~\cite[prop.\,G]{Gae}, sur l'ouvert d'irr\'eductibilit\'e absolue $U$ de $X$,
on dispose d'une repr\'esentation \`a valeurs dans une alg\`ebre d'Azumaya $A$ sur $\O(U)$.
Par ailleurs, d'apr\`es~\cite[cor.\,2.23]{Gae}, le compl\'et\'e de $\O(U)$ en $x$,
est l'anneau des d\'eformations universelles de $V_0$ de d\'eterminant fix\'e. 
Comme $V_0$ est absolument irr\'eductible, cet anneau est isomorphe \`a $L[[T_1,T_2,T_3]]$;
en particulier $U$ est lisse en tout point. L'alg\`ebre $A$ se trivialise sur une extension finie \'etale
de $\O(U)$ et, comme $U$ est lisse en $x$, cette extension \'etale est triviale sur une petite boule
autour de $x$.}
 d'une repr\'esentation
$\rho_B:\G_{\Q_p}\to \gl_2(\O(B)^+)$.
Alors $\O(B)^+\cong \O_{L}[[T_1,T_2,T_3]]$ (mais il ne semble pas
y avoir de choix naturel des coordonn\'ees $T_1,T_2,T_3$).

Les techniques des ${\rm n}^{\rm os}$~II.2.4 et~II.3.1 de~\cite{gl2} fournissent
un faisceau $G$-\'equivariant $U\mapsto D(\rho_B)\boxtimes U$ sur $\piqp(\Q_p)$
et une $\O(B)^+[\frac{1}{p}]$-repr\'esentation $\bPi(\rho_B)$ de $G$ telle que, 
pour tout $y\in U$, on ait $\bPi(\rho_B)_y=\bPi(\rho_y)$,
et qui vit dans une suite exacte\footnote{$\bPi(\rho_B)^\dual:={\rm Hom}_{\O(B)^+}(\bPi(\rho_B),\O(B)^+[\frac{1}{p}])$}
$$0\to \bPi(\rho_B)^\dual\otimes((x|x|)^{-1}\delta)\to D(\rho_B)\boxtimes\piqp(\Q_p)\to \bPi(\rho_B)\to 0$$
 Si $E$ est un quotient de $\O(B)^+[\frac{1}{p}]$,
de dimension finie sur $L$, alors $\rho_E:=E\otimes \rho_B$ est une $E$-repr\'esentation de dimension~$2$,
et $\bPi(\rho_E):=E\otimes\bPi(\rho_B)$ est une $E$-repr\'esentation de $G$, v\'erifiant
${\bf V}(\bPi(\rho_E))=\rho_E$. En sp\'ecialisant la suite exacte ci-dessus et en utilisant
les techniques de~\cite[chap.~V]{gl2} ou de~\cite[chap.\,V, VI]{CD}, on obtient des suites
exactes
\begin{align*}
&0\to \bPi(\rho_E)^\dual\otimes((x|x|)^{-1}\delta)\to D(\rho_E)\boxtimes\piqp(\Q_p)\to \bPi(\rho_E)\to 0\\
&0\to (\bPi(\rho_E)^{\rm an})^\dual\otimes((x|x|)^{-1}\delta)\to 
D_{\rm rig}(\rho_E)\boxtimes\piqp(\Q_p)\to \bPi(\rho_E)^{\rm an}\to 0
\end{align*}
On peut appliquer ceci, en particulier, aux quotients $E$ de 
$R^{{\rm ps},\delta}_{\cal B}[\frac{1}{p}]/{\goth m}_x^n$, pour $n\in\N$.
Les $E$-repr\'esentations $V=\rho_E$ que l'on obtient v\'erifient ${\rm End}_{E[\G_{\Q_p}]}V=E$ et, 
vues comme $L$-repr\'esentations,
 n'ont alors que $V_0$ comme composante de Jordan-H\"older.

\smallskip
Soit $V$ comme ci-dessus ($E$ est donc une alg\`ebre locale de corps r\'esiduel $L$;
on note ${\goth m}_E$ son id\'eal maximal).

\begin{theo}\label{defo2}
$\bPi(V)^{\rm alg}$ est dense dans $\bPi(V)$ si et seulement si
$V$ est de Rham \`a poids de Hodge-Tate non tous \'egaux.
\end{theo}
\begin{rema}\label{defo3}
(i) Si $V=V_0$, ce r\'esultat est
prouv\'e dans \cite[chap.\,VI]{gl2} (avec des preuves simplifi\'ees dans~\cite{Dosp1,poids});
si $V$ est une extension de $V_0$ par $V_0$, il est prouv\'e dans~\cite{D-dR}.

(ii) Soit $\Pi$ une $L$-repr\'esentation unitaire de $G$ dont toutes les composantes
de Jordan-H\"older sont isomorphes \`a $\Pi_0$. Supposons $\Pi_0^{\rm alg}$ irr\'eductible et
dense dans $\Pi_0$. Alors une r\'ecurrence imm\'ediate 
(utilisant la densit\'e de $(\Pi/\Pi_0)^{\rm alg}$ dans $\Pi/\Pi_0$ si $\Pi^{\rm alg}$
est dense dans $\Pi$) sur la longueur de $\Pi$ montre
que $\Pi^{\rm alg}$ est dense dans $\Pi$ si et seulement si ${\rm lg}(\Pi^{\rm alg})={\rm lg}(\Pi)$,
et que, si c'est le cas, ${\rm lg}(W^{\rm alg})={\rm lg}(W)$  pour tout sous-quotient $W$ de $\Pi$.
\end{rema}

\begin{proof}
La preuve du cas $V=V_0$ que l'on trouve
 dans~\cite{poids} s'\'etend presque verbatim au cas $n$ quelconque. Nous allons en esquisser
les grandes lignes.

Le cas $V=V_0$ permet de supposer $V_0$ de Rham, \`a poids de Hodge-Tate $k_1<k_2$.
Soient $\Delta:=D_{\rm rig}(V)$ et $\Delta_0:=D_{\rm rig}(V_0)$.
Comme ${\rm End}\,V=E$, on a
${\rm End}\,\Delta=E$, ce qui 
permet d'adapter la preuve de~\cite[prop.\,2.2]{poids}
pour obtenir $A_1,A_2\in E$ v\'erifiant $A_1+A_2=k_1+k_2$ (car $\det V=\delta$)
et $(\nabla-A_1)(\nabla-A_2)\Delta\subset t\Delta$ (on a donc
$A_i=k_i+P_i$ avec $P_i\in {\goth m}_E$, si $i=1,2$).
La preuve de \cite[th.\,2.15]{poids} montre que le module qui permet
de d\'ecrire 
(au moins en tant que module sous l'action du borel) 
les vecteurs localement alg\'ebriques 
de $\bPi(V)$ est $(\Delta_{\rm dif}^-)^{U({\goth g}){\text{-fini}}}$;
plus pr\'ecis\'ement, en adaptant cette preuve, on montre que le
$L_\infty[t]$-module $(\Delta_{\rm dif}^-)^{U({\goth g}){\text{-fini}}}$
est de longueur $(k_2-k_1){\rm lg}(\bPi(V)^{\rm alg})$.  La preuve de~\cite[prop.\,2.7]{poids}
montre alors que ${\rm lg}_{L_\infty[t]}((\Delta_{\rm dif}^-)^{U({\goth g}){\text{-fini}}})$
est la longueur du quotient par le plus grand sous-$L_\infty[[t]]$-r\'eseau $N$ de $\Delta^+_{\rm dif}$
v\'erifiant $(\nabla-A_1)(\nabla-A_2)N\subset tN$.

Il existe une base $e_{0,1},e_{0,2}$ de $\Delta_{0,{\rm dif}}^+$ v\'erifiant
$\nabla(e_{0,1})=k_1e_{0,1}$ et $\nabla e_{0,2}=k_2e_{0,2}$.
Les seuls sous-$L_\infty[[t]]$-modules d'indice fini de $\Delta_{0,{\rm dif}}^+$ v\'erifiant
$(\nabla-k_1)(\nabla-k_2)N\subset tN$ sont $\Delta_{0,{\rm dif}}^+$ et 
le module $N_{0,{\rm dif}}$ engendr\'e par $e_{0,2}$ et $t^{k_2-k_1}e_{0,1}$.

On peut relever $e_{0,1},e_{0,2}$ en une base $e_1,e_2$
de $\Delta_{{\rm dif}}^+$ v\'erifiant
$\nabla(e_{1})=A_1e_{1}$ 
et $\nabla(e_{2})=A_2e_{2}+Bt^{k_2-k_1}e_1$
avec $B\in L_\infty\otimes_L{\goth m}_E$. 
La repr\'esentation $V$ est de Rham si et seulement si $A_1=k_1$, $A_2=k_2$ et $B=0$.

Posons $L_{E,\infty}:= L_\infty\otimes E$.
Les seuls sous-$L_{E,\infty}[[t]]$-modules d'indice fini de $\Delta_{{\rm dif}}^+$ v\'erifiant
$(\nabla-A_1)(\nabla-A_2)N\subset (t,{\goth m}_E)N$ sont
les $N_{\rm dif}+ \Lambda\Delta_{{\rm dif}}^+$, o\`u
$N_{{\rm dif}}$ est le module engendr\'e par $e_{2}$ et $t^{k_2-k_1}e_{1}$ et $\Lambda$ est 
un id\'eal de $E$.
Parmi ceux-ci, le seul qui peut fournir des vecteurs localement alg\'ebriques de la bonne
longueur est $N_{\rm dif}$, i.e.~$\bPi(V)^{\rm alg}$ est dense dans $\bPi(V)$
si et seulement si  $(\nabla-A_1)(\nabla-A_2)N_{\rm dif}\subset tN_{\rm dif}$.
Cette derni\`ere condition \'equivaut \`a $(\nabla-A_2)t^{k_2-k_1}e_1\in tN_{\rm dif}$
et $(\nabla-A_2)e_2\in tN_{\rm dif}$. 

\quad $\bullet$ La premi\`ere des deux conditions \'equivaut
\`a 
$A_1-A_2+k_2-k_1=0$, et comme $A_1+A_2=k_1+k_2$, cela \'equivaut \`a $A_1=k_1$ et $A_2=k_2$.

\quad $\bullet$ La seconde \'equivaut \`a $B=0$.

La conjonction des deux conditions \'equivaut donc \`a ce que $V$ soit de Rham,
ce qui permet de conclure.
\end{proof}

\Subsection{D\'eformations infinit\'esimales de $V_{M,{\cal L}}^{[a,b]}$}\label{defo4}
Soit $M_{\rm dR}$ le $L$-module de rang~$2$ d\'efini par:
$$M_{\rm dR}:=(\Qbar_p\otimes_{\Q_p}M)^{\G_{\Q_p}}$$
Fixons ${\cal L}\in\piqp(M_{\rm dR})(\Qbar_p)$ et soit $L_{\cal L}$
le corps de d\'efinition de ${\cal L}$. 
Choisissons une base $e_1,e_2$ de $M_{\rm dR}$ sur $L$, 
avec $e_1\notin{\cal L}$.
Alors il existe $z({\cal L})\in\Qbar_p$ tel que ${\cal L}$ soit la droite
engendr\'ee par $e_2+z({\cal L}) e_1$, 
et $L_{\cal L}=L(z({\cal L}))$.

Si $n\geq 0$, soit $M^{[a,b]}_{{\cal L},n}$ le $(\varphi,N,\G_{\Q_p})$-module filtr\'e,
de rang $2$ sur $L_{\cal L}[T]/T^{n+1}$, d\'efini par
$$M^{[a,b]}_{{\cal L},n}=(L_{\cal L}[T]/T^{n+1})\otimes M$$
 en tant que $(\varphi,N,\G_{\Q_p})$-module,
la filtration ${\rm Fil}^\bullet_{{\cal L},n}$
sur $$(\Qbar_p\otimes M^{[a,b]}_{{\cal L},n})^{\G_{\Q_p}}=(L_{\cal L}[T]/T^{n+1})\otimes M_{\rm dR}$$
\'etant d\'efinie par
$${\rm Fil}^i_{{\cal L},n}=
\begin{cases} (L_{\cal L}[T]/T^{n+1})\otimes M_{\rm dR} &{\text{si $i\leq -b$,}}\\
(L_{\cal L}[T]/T^{n+1})\cdot (e_2+(z({\cal L})+T)e_1) &{\text{si $-b+1\leq i\leq -a$,}}\\
0 & {\text{si $i\geq -a+1$.}}
\end{cases}$$
Notons que le sous-module $T^kM^{[a,b]}_{{\cal L},n}$ de $M^{[a,b]}_{{\cal L},n}$
est isomorphe \`a $M^{[a,b]}_{{\cal L},n-k}$ comme $(\varphi,N,\G_{\Q_p})$-module filtr\'e
(sur $L_{\cal L}[T]/T^{n+1}$ et donc a fortiori sur $L$).
\begin{lemm}\label{defo250}
Si $n\geq 1$, alors:

{\rm (i)}
$M^{[a,b]}_{{\cal L},n}$, vu comme $L_{\cal L}$-$(\varphi,N,\G_{\Q_p})$-module filtr\'e,
 est faiblement admissible. 

{\rm (ii)} Les seuls sous-$L_{\cal L}$-$(\varphi,N,\G_{\Q_p})$-modules filtr\'es
faiblement admissibles de $M^{[a,b]}_{{\cal L},n}$ sont les $T^kM^{[a,b]}_{{\cal L},n}$, 
pour $0\leq k\leq n$.
\end{lemm}
\begin{proof}
Quitte \`a remplacer $L$ par $L_{\cal L}$, et $e_2$ par $e_2+z({\cal L}) e_1$, on peut supposer
que ${\cal L}$ est d\'efinie sur $L$ et 
que $z({\cal L})=0$.

Les seuls sous-$L$-$(\varphi,N,\G_{\Q_p})$-modules de $M^{[a,b]}_{{\cal L},n}$ sont
les $\Lambda\otimes_L M$, o\`u $\Lambda$ est un sous-$L$-module de $L_{\cal L}[T]/T^{n+1}$.
Comme la filtration n'a que deux crans, le (i) \'equivaut \`a ce que
$$\dim_L\big((\Lambda\otimes_L M_{\rm dR})\cap {\rm Fil}^{-a}_{{\cal L},n}\big)\leq \dim_L\Lambda$$
pour tout $\Lambda$, et le (ii) \`a ce que les seuls $\Lambda$ pour lesquels il y a \'egalit\'e
sont les $T^k(L[T]/T^{n+1-k})$.
Comme ${\rm Fil}^{-a}_{{\cal L},n}$ est l'ensemble des $Pe_2+TPe_1$,
l'intersection ci-dessus est l'ensemble des $Pe_2+TPe_1$ avec $P\in\Lambda$ et $TP\in\Lambda$.
L'in\'egalit\'e dans le (i) est donc une \'evidence et, s'il y a \'egalit\'e, c'est
que $\Lambda$ est stable par $P\mapsto TP$, et donc est un id\'eal de $L[T]/T^{n+1-k}$,
et donc de la forme $T^k(L[T]/T^{n+1-k})$,
ce qui permet de conclure.
\end{proof}

Le (i) du lemme~\ref{defo250} fournit (gr\^ace \`a~\cite{CF}) une repr\'esentation
$$V_{M,{\cal L},n}^{[a,b]}:={\bf V}_{\rm st}(M^{[a,b]}_{{\cal L},n})$$
C'est une d\'eformation \`a $L_{\cal L}[T]/T^{n+1}$
de $V_{M,{\cal L}}^{[a,b]}$, qui est de Rham, dont la r\'eduction modulo $T^2$ est non scind\'ee
(car la filtration n'est pas d\'efinie sur $L_{\cal L}$ modulo $T^2$),
et dont, d'apr\`es le (ii) du lemme~\ref{defo250},
 les seuls sous-objets sont les $T^kV_{M,{\cal L},n}^{[a,b]}$ (isomorphe \`a 
$V_{M,{\cal L},n-k}^{[a,b]}$), pour $0\leq k\leq n$.
\begin{rema}\label{petit2}
(i) En adaptant les m\'ethodes de~\cite[\S\,5.1]{triangulines}, on peut construire
une d\'eformation de Rham de $V_{M,{\cal L}}^{[a,b]}$ dans un voisinage $p$-adique,
et pas juste dans un voisinage formel.

(ii) Le th.\,\ref{petit15} et le lemme~\ref{petit14} ci-dessous donnent une construction
d'une d\'eformation au-dessus de tout l'espace des repr\'esentations
 de Rham de type $(M,a,b)$.
\end{rema}

Soit ${\cal B}$ le bloc correspondant \`a la r\'eduction de $V_{M,{\cal L}}^{[a,b]}$.
Alors $V_{M,{\cal L}}^{[a,b]}$ d\'efinit un point de ${\rm Spec}\,R^{{\rm ps},\delta}_{\cal B}[\frac{1}{p}]$;
notons $\widehat R_{\cal L}$ le compl\'et\'e de $R^{{\rm ps},\delta}_{\cal B}[\frac{1}{p}]$ en l'id\'eal
maximal correspondant \`a ce point (isomorphe \`a $L_{\cal L}[[T_1,T_2,T_3]]$), 
et ${\goth m}_{\cal L}$
son id\'eal maximal.
On dispose d'une repr\'esentation $\widehat\rho_{\cal L}:\G_{\Q_p}\to \gl_2(\widehat R_{\cal L})$
universelle pour les d\'eformations de $V_{M,{\cal L}}^{[a,b]}$ aux $L_{\cal L}$-alg\`ebres locales
finies sur $L_{\cal L}$ de corps r\'esiduel $L_{\cal L}$: 
si $E$ est une telle alg\`ebre et si $\rho_E:\G_{\Q_p}\to \gl_2(E)$ a pour repr\'esentation
r\'esiduelle $V_{M,{\cal L}}^{[a,b]}$, alors il existe $\lambda:\widehat R_{\cal L}\to E$ tel que
$\rho_E\cong E\otimes \widehat\rho_{\cal L}$.  De plus,
${\goth m}_{{\cal L}}/{\goth m}_{{\cal L}}^2$ est isomorphe
au groupe des extensions de $V_{M,{\cal L}}^{[a,b]}$ par $V_{M,{\cal L}}^{[a,b]}$ de d\'eterminant $\delta$
(une telle extension est naturellement une $(L_{\cal L}[T]/T^2)$-repr\'esentation de dimension~$2$
et son d\'eterminant est, a priori, \`a valeurs dans $L_{\cal L}[T]/T^2$).

On note $\widehat R_{{\cal L},{\rm dR}}$ le quotient de $\widehat R_{\cal L}$ classifiant
les repr\'esentations de Rham: on a $\widehat R_{{\cal L},{\rm dR}}=\widehat R_{\cal L}/I$,
o\`u $I=\cap\,{\goth a}$ et ${\goth a}$ parcourt les id\'eaux de $\widehat R_{\cal L}$ tels
que $\widehat R_{\cal L}/{\goth a}$ soit de dimension finie sur $L_{\cal L}$
(ce qui \'equivaut \`a ${\goth a}\supset {\goth m}_{{\cal L}}^n$, pour $n\gg0$),
et $(\widehat R_{\cal L}/{\goth a})\otimes \widehat\rho_{\cal L}$ soit de Rham.

\begin{prop}\label{defo5.1}
$\widehat R_{{\cal L},{\rm dR}}\cong L_{\cal L}[[T]]$.
\end{prop}
\begin{proof}
Comme $V_{M,{\cal L},2}^{[a,b]}$ est, \`a isomorphisme pr\`es, l'unique \cite[prop.\,7.17]{D-dR}
extension non triviale de $V_{M,{\cal L}}^{[a,b]}$ par $V_{M,{\cal L}}^{[a,b]}$, de d\'eterminant $\delta$,
qui est de Rham, il s'ensuit que $\widehat R_{{\cal L},{\rm dR}}$ est un anneau local r\'egulier 
de dimension~$1$,
et donc est un quotient
de $L_{\cal L}[[T]]$. Par ailleurs, l'existence de $V_{M,{\cal L},n}^{[a,b]}$, pour tout $n$, 
fournit une suite compatible de morphismes $\widehat R_{{\cal L},{\rm dR}}\to L_{\cal L}[T]/T^{n+1}$,
pour $n\in\N$. Le fait que $V_{M,{\cal L},2}^{[a,b]}$ soit non scind\'ee implique que cette
fl\`eche est surjective pour $n=2$, et donc aussi pour tout $n$. 
Il s'ensuit 
que $\widehat R_{{\cal L},{\rm dR}}$ admet pour quotient $L_{\cal L}[T]/T^{n+1}$, pour tout $n$.
On en d\'eduit le r\'esultat.
\end{proof}

\begin{rema}\label{defo5}
On d\'eduit de la prop.\,\ref{defo5.1} que $V_{M,{\cal L},n}^{[a,b]}$ est, \`a isomorphisme pr\`es,
 l'unique d\'eformation de Rham de $V_{M,{\cal L}}^{[a,b]}$ \`a $L_{\cal L}[T]/T^{n+1}$, de d\'eterminant $\delta$,
dont le quotient par $T^2$ soit non scind\'e.
\end{rema}
\Subsection{D\'eformations infinit\'esimales de $\Pi_{M,{\cal L}}^{[a,b]}$}\label{defo6}
En appliquant les constructions du \S\,\ref{defo1} \`a $V=V^{[a,b]}_{M,{\cal L},n}$,
on construit une repr\'esentation 
$$\Pi^{[a,b]}_{M,{\cal L},n}:=\bPi(V^{[a,b]}_{M,{\cal L},n})$$
\begin{rema}\label{petit3}
(i)
Il r\'esulte du th.\,\ref{defo2} que $(\Pi^{[a,b]}_{M,{\cal L},n})^{\rm alg}$ est
dense dans $\Pi^{[a,b]}_{M,{\cal L},n}$ et on d\'eduit du (ii) de la rem.\,\ref{defo3}
que 
$$(\Pi^{[a,b]}_{M,{\cal L},n})^{\rm alg}\cong (L_{\cal L}[T]/T^{n+1})\otimes {\rm LL}^{[a,b]}(M)$$
En effet, les deux membres ont m\^eme longueur d'apr\`es la rem.\,\ref{defo3}, 
sont des $L_{\cal L}[T]/T^{n+1}$-modules, et $T^n$ induit un isomorphisme de $\Pi^{[a,b]}_{M,{\cal L},n}/T$
sur $T^n\Pi^{[a,b]}_{M,{\cal L},n}$ (tous les deux isomorphes \`a $\Pi^{[a,b]}_{M,{\cal L}}$).
Si on choisit ${\rm LL}^{[a,b]}(M)\to (\Pi^{[a,b]}_{M,{\cal L},n})^{\rm alg}$
tel que la compos\'ee avec la projection sur $(\Pi^{[a,b]}_{M,{\cal L},n}/T)^{\rm alg}$ soit non nul
(possible, d'apr\`es la rem.\,\ref{defo3}), alors
l'application
naturelle $(L_{\cal L}[T]/T^{n+1})\otimes {\rm LL}^{[a,b]}(M)\to (\Pi^{[a,b]}_{M,{\cal L},n})^{\rm alg}$
est injective (car injective sur $T^n$ d'apr\`es ce qui pr\'ec\`ede) 
et donc un isomorphisme pour des raisons de longueur.

(ii) En fait, si $\lambda\in (L_{\cal L}[T]/T^{n+1})$ est inversible, alors $\lambda\cdot {\rm LL}^{[a,b]}(M)$
est d\'ej\`a dense dans $\Pi^{[a,b]}_{M,{\cal L},n}$ car les seuls sous-objets stricts de
$\Pi^{[a,b]}_{M,{\cal L},n}$ sont les $T^k\Pi^{[a,b]}_{M,{\cal L},n}$ pour $1\leq k\leq n-1$
(cela se voit en appliquant le foncteur $\Pi\mapsto{\bf V}(\Pi)$ et en utilisant
le fait que les seuls sous-objets stricts de
$V^{[a,b]}_{M,{\cal L},n}$ sont les $T^kV^{[a,b]}_{M,{\cal L},n}$, cf.~lemme\,\ref{defo250} et
d\'efinition de $V^{[a,b]}_{M,{\cal L},n}$),
et $\lambda\cdot {\rm LL}^{[a,b]}(M)$ n'est inclus dans aucun.
Il r\'esulte donc de la prop.\,\ref{A12} que $\Pi^{[a,b]}_{M,{\cal L},n}$ est un quotient
de $\widehat{\rm LL}\hskip0mm^{[a,b]}(M)$.
\end{rema}

\begin{prop}\label{defo5.2}
Si $\Pi$ est une repr\'esentation unitaire de $G$, de caract\`ere central~$\delta$,
 munie d'un morphisme ${\rm LL}^{[a,b]}(M)\to\Pi$
d'image dense, et si les composantes de Jordan-H\"older de $\Pi$ sont $\Pi^{[a,b]}_{M,{\cal L}}$
avec multiplicit\'e~$n$, alors $\Pi\cong \Pi^{[a,b]}_{M,{\cal L},n}$.
\end{prop}
\begin{proof}
D'apr\`es~\cite[cor.\,6.16]{PT}, la cat\'egorie des repr\'esentations unitaires de longueur finie,
dont toutes les composantes de Jordan-H\"older sont $\Pi^{[a,b]}_{M,{\cal L}}$ est \'equivalence
\`a la cat\'egorie des $\widehat R_{\cal L}$-modules. On peut d\'ecrire le foncteur
r\'ealisant cette \'equivalence par la th\'eorie de Gabriel~\cite{gaby}.
Soit $J_{\cal L}$ l'enveloppe injective de $\Pi^{[a,b]}_{M,{\cal L}}$ dans la cat\'egorie des
limites inductives de repr\'esentations unitaires de $G$. Alors ${\rm End}\,J_{\cal L}=\widehat R_{\cal L}$
d'apr\`es~\cite[cor.\,6.16]{PT} et comme le bloc de $\Pi^{[a,b]}_{M,{\cal L}}$ dans la cat\'egorie
des repr\'esentations unitaires de $G$ n'a qu'un \'el\'ement, on a
$$\Pi\cong({\rm Hom}(J_{\cal L}^\dual,\Pi^\dual)\otimes_{\widehat R_{\cal L}}J_{\cal L}^\dual)^\dual$$
L'hypoth\`ese selon laquelle ${\rm LL}^{[a,b]}(M)\to\Pi$ est d'image dense implique 
a fortiori que $\Pi^{\rm alg}$ est dense dans $\Pi$. Il en r\'esulte, 
d'apr\`es le th.\,\ref{defo2} et la d\'efinition de $\widehat R_{{\cal L},{\rm dR}}$, que 
${\rm Hom}(J_{\cal L}^\dual,\Pi^\dual)$ est en fait un $\widehat R_{{\cal L},{\rm dR}}$-module.
La prop.\,\ref{defo5.1} implique alors que
$${\rm Hom}(J_{\cal L}^\dual,\Pi^\dual)\cong\oplus_iL_{\cal L}[[T]]/T^{n_i}
\quad{\rm et}\quad
\Pi\cong\oplus_{i\in I} \Pi^{[a,b]}_{M,{\cal L},n_i}$$
 o\`u $I$ est un ensemble fini.
On conclut en remarquant que
l'hypoth\`ese selon laquelle ${\rm LL}^{[a,b]}(M)\to\Pi$ est d'image dense implique
que $I$ n'a qu'un seul \'el\'ement: en effet, cette hypoth\`ese permet de d\'eduire que ${\rm Hom}_G(\Pi, \Pi^{[a,b]}_{M,{\cal L}})$ 
est de dimension au plus $1$, mais d'autre part cet espace est de dimension au moins $|I|$.
\end{proof}

\section{Compl\'etion ${\cal B}$-adique et anneaux de Kisin}\label{PETIT11}

\Subsection{Les $R_{M,{\cal B}}^{[a,b]}$-modules ${\rm LL}_{\cal B}^{[a,b]}(M)$ et $\rho^{[a,b]}_{M,{\cal B}}$}\label{PETIT12}
Notons ${I_M(\sigma^+)}_{\cal B}$ le compl\'et\'e ${\cal B}$-adique de $I_M(\sigma^+)$
(cf.~\S\,\ref{PETIT5}).
On d\'efinit alors le compl\'et\'e ${\cal B}$-adique de ${\rm LL}^{[a,b]}(M)$ par:
$${\rm LL}^{[a,b]}_{\cal B}(M):=L\otimes_{\O_L}{I_M(\sigma^+)}_{\cal B}$$
(le r\'esultat ne d\'epend pas du choix de $\sigma^+$).
Posons 
$$\delta=\delta_M^{[a,b]},\quad 
I^+:={I_M(\sigma^+)}_{\cal B}$$
Alors $I^+$ est une limite projective d'objets
$\Pi_i$ de ${\rm Tors}^{\delta}_{\cal B}\,G$.  On a donc une action de
$Z_{\cal B}^{\delta}$ sur $I^+$ et sur ${\rm LL}^{[a,b]}_{\cal B}(M)$.
Par fonctorialit\'e, cela fournit une action de $Z_{\cal B}^{\delta}$ sur
$\rho^{[a,b]}_{M,{\cal B}}$. Cette action co\"{\i}ncide avec l'action naturelle
de~$R^{{\rm ps},\delta}_{\cal B}$, via l'isomorphisme du~th.\,\ref{pasku6}.

On note $R_{M,{\cal B}}^{[a,b],+}$ le quotient \`a travers lequel $Z_{\cal B}^\delta$
agit sur $I^+$ (notons que $R_{M,{\cal B}}^{[a,b],+}$ est sans $p$-torsion puisque
$I^+$ l'est, cf.~cor.\,\ref{petit11}), et on pose
$$R_{M,{\cal B}}^{[a,b]}:=R_{M,{\cal B}}^{[a,b],+}[\tfrac{1}{p}]$$

\begin{rema}\label{A38}
On verra plus loin que $R_{M,{\cal B}}^{[a,b]}$ est l'anneau des fonctions analytiques
born\'ees sur un ouvert strict $U^{\rm an}_{M,{\cal B}}$ de 
la droite projective analytique ${\piqp}(M_{\rm dR})^{\rm an}$, et donc est tr\`es loin
d'\^etre r\'eduit \`a $L$. Or $R_{M,{\cal B}}^{[a,b]}$ agit fid\`element sur ${\rm LL}_{\cal B}^{[a,b]}(M)$,
ce qui contraste avec le fait que
${\rm End}_{L[G]}\widehat{\rm LL}\hskip0mm^{[a,b]}(M)=L$ (prop.\,\ref{A32}).

On peut envisager les choses de la mani\`ere suivante (et il est probable que la th\'eorie de
Dotto-Emerton-Gee~\cite{DGE,DGE2} rende cette vision correcte).
Tous les objets
ci-dessus {\og vivent sur ${\piqp}(M_{\rm dR})$\fg}, i.e.~$\widehat{\rm LL}\hskip0mm^{[a,b]}(M)$ est un faisceau sur
${\piqp}(M_{\rm dR})$
et ${\rm LL}\hskip0mm^{[a,b]}_{\cal B}(M)$ est le localis\'e-compl\'et\'e
de $\widehat{\rm LL}\hskip0mm^{[a,b]}(M)$ en l'ouvert $U^{\rm an}_{M,{\cal B}}$.
Alors 
${\rm End}_{L[G]}\widehat{\rm LL}\hskip0mm^{[a,b]}(M)$ est un fibr\'e 
en droites sur ${\piqp}(M_{\rm dR})$, qui est trivial ce qui explique que 
${\rm End}_{L[G]}\widehat{\rm LL}\hskip0mm^{[a,b]}(M)=L$ (sections globales constantes).
\end{rema}

Posons 
$$\rho^{[a,b]}_{M,{\cal B}}:=
{\bf V}({\rm LL}^{[a,b]}_{\cal B}(M))$$
Alors $R_{M,{\cal B}}^{[a,b],+}$ agit, par fonctorialit\'e de $\Pi\mapsto{\bf V}(\Pi)$,
 sur ${\bf V}(I^+)$ et $R_{M,{\cal B}}^{[a,b]}$
agit sur $\rho^{[a,b]}_{M,{\cal B}}$.
\begin{lemm}\label{petit13}
$\rho^{[a,b]}_{M,{\cal B}}$ est de type fini sur $R_{M,{\cal B}}^{[a,b]}$.
\end{lemm}
\begin{proof}
$I_M(\sigma^+)$ est un facteur direct de $I(\sigma^+)$ et donc $I^+$ est un facteur
direct de $I(\sigma^+)_{\cal B}$. Il suffit donc de prouver que
${\bf V}(I(\sigma^+)_{\cal B})$
est de type fini sur $R^{{\rm ps},\delta}_{\cal B}$.
On a $k_L\otimes{\bf V}( I(\sigma^+)_{\cal B})={\bf V}(k_L\otimes  I(\sigma^+)_{\cal B})
={\bf V}({I(k_L\otimes\sigma^+)}_{\cal B})$ (les morphismes de transitions entre les objets
entrant dans la d\'efinition des compl\'et\'es ${\cal B}$-adiques sont surjectifs, et donc les
${\rm R}^1\lim$ s'annulent). Maintenant $I(k_L\otimes\sigma^+)$ est une extension successive finie
d'induites compactes $I(W)$, o\`u les $W$ sont de la forme $W_{r,\chi}$,
 et le cor.~\ref{A34} implique que $k_L\otimes{\bf V}( I(\sigma^+)_{\cal B})$ est de type fini 
sur $R^{{\rm ps},\delta}_{\cal B}$.
Il en est donc de m\^eme de ${\bf V}( I(\sigma^+)_{\cal B})$.
\end{proof}

\Subsection{Sp\'ecialisation en un point de ${\rm Spec}\,R_{M,{\cal B}}^{[a,b]}$}\label{PETIT13}
On note $U_{M,{\cal B}}$ l'ensemble des ${\cal L}\in{\piqp}(M_{\rm dR})(\Qbar_p)$ tels
que $\Pi_{M,{\cal L}}\in {\rm Ban}^\delta_{\cal B}\,G$; c'est aussi l'ensemble
des ${\cal L}$ tels que $V_{M,{\cal L}}$ ait pour r\'eduction $\rho_{\cal B}$. 
Comme nous le verrons,
$U_{M,{\cal B}}$ est l'ensemble des points classiques d'un ouvert analytique
$U_{M,{\cal B}}^{\rm an}$ de la droite projective ${\piqp}(M_{\rm dR})^{\rm an}$.
Soient 
$$X={\rm Spec}\,R^{{\rm ps},\delta}_{\cal B}[\tfrac{1}{p}],\quad X_M={\rm Spec}\,R_{M,{\cal B}}^{[a,b]}$$
On note aussi $X^{\rm an}$ et $X_M^{\rm an}$ les espaces analytiques associ\'es
\`a $X$ et $X_M$: $X^{\rm an}$ est de dimension~$3$ et, comme nous le verrons,
$X_M^{\rm an}$ est de dimension~$1$ (isomorphe \`a l'ouvert $U_{M,{\cal B}}^{\rm an}$
de ${\piqp}(M_{\rm dR})^{\rm an}$).

On note ${\rm Ban}^\delta\,G$ la cat\'egorie des $L[G]$-banachs unitaires
de caract\`ere central $\delta$, de longueur finie.
On note ${\rm Ban}^\delta_{\cal B}\,G$ la sous-cat\'egorie
de ${\rm Ban}^\delta\,G$ des $\Pi$ dont la r\'eduction est un objet de ${\rm Tors}_{\cal B}^\delta$.
D'apr\`es {\paskunas}~\cite{Paskext}, tout objet $\Pi$ de ${\rm Ban}^\delta_{\cal B}\,G$
a une d\'ecomposition $\Pi=\oplus_x\Pi_x$, o\`u la somme porte sur $x\in X$ ferm\'e,
$\Pi_x=0$ sauf pour un nombre fini de $x$ et toutes les composantes
de Jordan-H\"older de $\Pi_x$ sont \'el\'ement du bloc ${\cal B}_x$ de
${\rm Ban}^\delta\,G$ correspondant \`a $x$ (si le pseudo-caract\`ere de $\G_{\Q_p}$
correspondant \`a $x$ est le caract\`ere d'une repr\'esentation irr\'eductible, alors
${\cal B}_x$ a un unique \'el\'ement, si ce pseudo-caract\`ere est $\chi_1\oplus\chi_2$,
ce bloc est constitu\'e des composantes de Jordan-H\"older des induites continues de
$\chi_1\otimes\chi_2\epsilon^{-1}$ et $\chi_2\otimes\chi_1\epsilon^{-1}$ -- ces repr\'esentations
sont toutes les deux irr\'eductibles sauf si $\chi_2=\chi_1\epsilon^{\pm1}$).

Si $x\in X_M$, notons ${\goth m}_x$, l'id\'eal maximal de $R_{M,{\cal B}}^{[a,b]}$ qui lui est associ\'e,
et $L_x $ le corps r\'esiduel $R_{M,{\cal B}}^{[a,b]}/{\goth m}_x$.
\begin{lemm}\label{petit14}
{\rm (i)} Si $x\in X_M$, il existe ${\cal L}(x)\in U_{M,{\cal B}}$ tel que
$$L_x \otimes_{R_{M,{\cal B}}^{[a,b]}} {\rm LL}^{[a,b]}_{\cal B}(M)\cong\Pi_{M,{\cal L}(x)},\quad
L_x \otimes_{R_{M,{\cal B}}^{[a,b]}} \rho^{[a,b]}_{M,{\cal B}}\cong V_{M,{\cal L}(x)}$$

{\rm (ii)} Plus g\'en\'eralement, si $n\geq 0$, alors
$$(R_{M,{\cal B}}^{[a,b]}/{\goth m}_x^{n+1})\otimes_{R_{M,{\cal B}}^{[a,b]}}{\rm LL}^{[a,b]}_{\cal B}(M)
\cong\Pi_{M,{\cal L}(x),n},
\quad
(R_{M,{\cal B}}^{[a,b]}/{\goth m}_x^{n+1})\otimes_{R_{M,{\cal B}}^{[a,b]}}\rho^{[a,b]}_{M,{\cal B}}
\cong V_{M,{\cal L}(x),n}$$

{\rm (iii)} $x\mapsto {\cal L}(x)$ est une bijection de $X_M$ sur $U_{M,{\cal B}}$.
\end{lemm}
\begin{proof}
Par construction, ${\rm LL}^{[a,b]}(M)$ est dense dans $\Pi_x=L_x \otimes {\rm LL}^{[a,b]}_{\cal B}(M)$.
Donc le bloc de ${\rm Ban}^\delta_{\cal B}\,G$ correspondant \`a $x$ contient un
unique $\Pi_{M,{\cal L}(x)}$, avec ${\cal L}(x)\in {\piqp}(M_{\rm dR})(L_x )$.
Comme $V_{M,{\cal L}(x)}$ est irr\'eductible,
ce bloc est r\'eduit \`a $\Pi_{M,{\cal L}(x)}$, et 
on peut d\'eduire les (i) et (ii) de la prop.\,\ref{defo5.2}.

Enfin, le (iii) est une cons\'equence de ce qui pr\'ec\`ede et de ce que ${\cal L}\mapsto
V_{M,{\cal L}}$ est injective.
\end{proof}

\Subsection{La propri\'et\'e universelle de $\rho_{M,{\cal B}}^{[a,b]}$}\label{PETIT14}
\begin{theo}\label{petit15}
{\rm (i)} $\rho_{M,{\cal B}}^{[a,b]}$ est localement libre sur $X_M$, de rang~$2$.

{\rm (ii)} $X_M$ est lisse, r\'eduit, purement de dimension~$1$.
\end{theo}
\begin{proof}
Posons, pour simplifier, $R:=R_{M,{\cal B}}^{[a,b]}$, $\rho:=\rho_{M,{\cal B}}^{[a,b]}$.
Alors $R$
agit fid\`element
sur ${\rm LL}^{[a,b]}_{\cal B}(M)$ et donc aussi sur $\rho$ car
${\bf V}$ ne tue que les repr\'esentations de dimension finie, et aucun des blocs
qui apparaissent n'en contient.

Notons $\widehat\rho_x$ le compl\'et\'e du localis\'e de $\rho$ en $x$.
Comme $\rho$ est de type fini sur $R$, on a une injection $R$-lin\'eaire
$\rho\hookrightarrow\prod_x\widehat\rho_x$, et comme $R$ agit fid\`element sur $\rho$,
il agit fid\`element sur $\prod_x\widehat\rho_x$. Il r\'esulte du (ii) du lemme~\ref{petit14}
que $\widehat\rho_x$ est libre de rang $2$ sur $L_x [[T_x]]$ (car  $V_{M,{\cal L}(x),n}$ est
libre de rang 2 sur $L_x [T_x]/T_x^{n+1}$), et comme $\rho_x$ est irr\'eductible, on a
${\rm End}_{\G_{\Q_p}}(\widehat\rho_x)\cong L_x [[T_x]]$.
On en d\'eduit une injection d'anneaux $R\hookrightarrow \prod_x L_x [[T_x]]$, ce qui prouve que $R$
est r\'eduit.

Comme $\rho_x$ est de rang~$2$, pour tout $x$ (lemme~\ref{petit14} (i)), 
et que $\rho$ est de type fini sur $R$,
cela implique que $\rho$ est localement libre, de rang~$2$, sur $X_M$.
On en d\'eduit, en utilisant ce qui pr\'ec\`ede, que
le compl\'et\'e $\widehat R_x$ de l'anneau local de $X_M$ en $x$ est 
$L_x [[T_x]]$, ce qui prouve que $X_M$ est lisse, purement de dimension~$1$.
\end{proof}

Il r\'esulte de~\cite[\S\,5.3]{BC} que l'application $x\mapsto {\cal L}(x)$ ci-dessus
est la restriction aux points classiques d'une application analytique
$X_M^{\rm an}\to {\piqp}(M_{\rm dR})^{\rm an}$. L'ensemble $U_{M,{\cal B}}$ ci-dessus
est donc l'ensemble des points classiques d'un ouvert analytique $U_{M,{\cal B}}^{\rm an}$
de ${\piqp}(M_{\rm dR})^{\rm an}$, et $x\mapsto{\cal L}(x)$ identifie
$X_M^{\rm an}$ \`a $U_{M,{\cal B}}^{\rm an}$ et $R_{M,{\cal B}}^{[a,b]}$ \`a l'anneau
des fonctions analytiques born\'ees sur $U_{M,{\cal B}}^{\rm an}$. Comme
$R_{M,{\cal B}}^{[a,b]}$ est un quotient de $R^{{\rm ps},\delta}_{\cal B}$ qui est de type fini,
$U_{M,{\cal B}}^{\rm an}$ n'a qu'un nombre fini de composantes connexes, et
comme un ouvert analytique connexe de ${\piqp}$ est un
disque ouvert priv\'e d'un nombre fini
de disques ferm\'es, 
on en d\'eduit les r\'esultats suivants.
\begin{coro}\label{petit16}
{\rm (i)} $R_{M,{\cal B}}^{[a,b]}$ est un produit fini d'anneaux principaux.

{\rm (ii)} $\rho_{M,{\cal B}}^{[a,b]}$ est libre sur $R_{M,{\cal B}}^{[a,b]}$, de rang~$2$.
\end{coro}

\begin{coro}\label{petit16.1}
$R_{M,{\cal B}}^{[a,b]}={\rm End}_G\,{\rm LL}_{\cal B}^{[a,b]}(M)$.
\end{coro}
\begin{proof}
L'inclusion $R_{M,{\cal B}}^{[a,b]}\subset{\rm End}_G\,{\rm LL}_{\cal B}^{[a,b]}(M)$ est imm\'ediate; montrons
l'inclusion dans l'autre sens. Soit donc $\alpha\in {\rm End}_G\,{\rm LL}_{\cal B}^{[a,b]}(M)$.
Alors $\alpha$ commute \`a l'action de $R_{M,{\cal B}}^{[a,b]}$ puisque puisque
$R_{M,{\cal B}}^{[a,b]}$ est un quotient de $Z_{\cal B}^\delta$ qui est le centre de la cat\'egorie.
Donc
${\bf V}(\alpha)\in {\rm End}_{\G_{\Q_p}}\,\rho_{M,{\cal B}}^{[a,b]}$ commute aussi
\`a l'action de $R_{M,{\cal B}}^{[a,b]}$ et puisque $\rho_{M,{\cal B}}^{[a,b]}$ est un
$R_{M,{\cal B}}^{[a,b]}$-module libre de rang~$2$, et toutes les sp\'ecialisations
de $\rho_{M,{\cal B}}^{[a,b]}$ sont irr\'eductibles, on en d\'eduit que ${\bf V}(\alpha)\in
{\rm M}_2(R_{M,{\cal B}}^{[a,b]})$ et que l'image de ${\bf V}(\alpha)$ mod ${\goth m}_x$ est une homoth\'etie
pour tout $x\in X_M$. Donc ${\bf V}(\alpha)$ est une homoth\'etie,
i.e.~${\bf V}(\alpha)\in R_{M,{\cal B}}^{[a,b]}$.  

On peut donc voir
$\alpha-{\bf V}(\alpha)$ comme un \'el\'ement de ${\rm End}_G\,{\rm LL}_{\cal B}^{[a,b]}(M)$,
et on a ${\bf V}(\alpha-{\bf V}(\alpha))=0$, et donc $\alpha-{\bf V}(\alpha)=0$, ce qui permet de
conclure.
\end{proof}

\begin{coro}\label{petit31}
Si $E$ est un quotient
de $R_{M,{\cal B}}^{[a,b]}$ de degr\'e fini sur $L$, alors $E\otimes \rho_{M,{\cal B}}^{[a,b]}$ est
une $E$-repr\'esentation de $\G_{\Q_p}$ de r\'eduction\footnote{Par d\'efinition,
c'est la semi-simplifi\'ee de $k_L\otimes_{\O_L}\Lambda$ (vue comme $k_L$-repr\'esentation de $\G_{\Q_p}$),
o\`u $\Lambda$ est n'importe quel $\O_L$-r\'eseau stable par $\G_{\Q_p}$.}
 $\rho_{\cal B}^{\oplus [L:E]}$, 
potentiellement semi-stable \`a poids $a$ et $b$,
dont le $D_{\rm pst}$ est $E\otimes M$.  
\end{coro}
\begin{proof}
Cela r\'esulte du (ii) du lemme~\ref{petit14}
 et du th\'eor\`eme des restes chinois.
\end{proof}

R\'eciproquement, on a le r\'esultat suivant
qui montre que $\rho_{M,{\cal B}}^{[a,b]}$ est universelle pour ces propri\'et\'es.

\begin{theo}\label{petit17}
Si $E$ est une $L$-alg\`ebre commutative de dimension $d$, 
et si $\rho:\G_{\Q_p}\to \gl_2(E)$ a pour r\'eduction $\rho_{\cal B}^{\oplus d}$ et d\'eterminant $\delta$,
est potentiellement semi-stable \`a poids $a$ et $b$, et si $D_{\rm pst}(\rho)=
E\otimes M$, alors il existe un morphisme $R_{M,{\cal B}}^{[a,b]}\to E$
tel que $\rho=E\otimes \rho_{M,{\cal B}}^{[a,b]}$.
\end{theo}
\begin{proof}
Si $E$ est un corps, alors $\rho$ est de la forme $E\otimes_{L_x }V_{M,{\cal L}(x)}^{[a,b]}$,
avec $x\in U_{M,{\cal B}}$, et comme $V_{M,{\cal L}(x)}^{[a,b]}=L_x \otimes \rho_{M,{\cal B}}^{[a,b]}$,
on obtient le r\'esultat dans ce cas.

Si $E=E_0[I]/T^k$, o\`u $E_0$ est un corps, alors $\rho_i:=(T^iE/T^{i+1}E)\otimes_E\rho$
est de la forme $E_0\otimes_{L_x }V_{M,{\cal L}(x)}^{[a,b]}$,
avec $x\in U_{M,{\cal B}}$ ind\'ependant de $i$.
L'extension 
$$0\to \rho_i\to (T^{i-1}E/T^{i+1}E)\otimes_E\rho\to \rho_{i-1}\to 0$$ 
de
$E_0\otimes_{L_x }V_{M,{\cal L}(x)}^{[a,b]}$ par elle-m\^eme est ind\'ependante de $i$
car la multiplication par $T$ sur $\rho$ induit un isomorphisme permettant de passer de $i$ \`a $i+1$.
Comme le groupe des extensions de Rham est de dimension~$1$, 
il y a deux cas possibles:

-- toutes ces extensions sont triviales et $\rho=E\otimes_{E_0}\rho_0$, et on conclut comme ci-dessus.

-- ces extensions sont non triviales et $\rho=E_0\otimes_{L_x }V_{M,{\cal L}(x),k}^{[a,b]}$,
et on d\'eduit le r\'esultat dans ce cas de ce que
$V_{M,{\cal L}(x),k}^{[a,b]}=
(R_{M,{\cal B}}^{[a,b]}/{\goth m}_x^k)\otimes \rho_{M,{\cal B}}^{[a,b]}$.

Si $E$ est un quotient de $E_0[T_1,\dots,T_r]/(T_1,\dots,T_r)^k$, en raisonnant comme
ci-dessus et en utilisant de nouveau que le groupe des extensions de Rham est de dimension~$1$,
on prouve que $\rho=E\otimes_{L_x }V_{M,{\cal L}(x),k}^{[a,b]}$
ou bien $\rho=E\otimes_{E'}\rho'$, avec $E'$ quotient de $E$ de la forme $E_0[I]/T^k$
et $\rho'\cong E_0\otimes_{L_x }V_{M,{\cal L}(x),k}^{[a,b]}$.

Le cas g\'en\'eral s'en d\'eduit en d\'ecomposant $E$ comme un produit d'alg\`ebres
quotients de $E_0[T_1,\dots,T_r]/(T_1,\dots,T_r)^k$.
\end{proof}

\Subsection{Applications \`a la conjecture de Breuil-M\'ezard}
Dans l'\'enonc\'e g\'eom\'etrique de la conjecture de Breuil-M\'ezard~\cite{BM,EG}, au lieu de fixer le
$D_{\rm pst}$, on fixe seulement la restriction \`a l'inertie et le d\'eterminant, ce qui fournit
a priori deux $D_{\rm pst}$ possibles puisqu'on peut tordre par le caract\`ere non ramifi\'e $\mu_{-1}$
d'ordre $2$ (rien n'emp\^eche que ces deux $D_{\rm pst}$ soient, en fait, isomorphes).
Le r\'esultat est, qu'au lieu de consid\'erer le sous-objet ${\rm LL}^{[a,b]}(M)$
de ${\rm ind}_{KZ}^G\sigma_M^{[a,b]}$ (cf.~\S\,\ref{PETIT5}), on consid\`ere \`a la place
${\rm ind}_{KZ}^G\sigma_M^{[a,b]}$ en entier.

On pose donc
\begin{align*}
M'=M &&{\rm ou}&&&M'=M\oplus (M\otimes\mu_{-1})\\
{\rm LL}_{\cal B}^{[a,b]}(M')={\rm LL}_{\cal B}^{[a,b]}(M)&&{\rm ou}&&&{\rm LL}_{\cal B}^{[a,b]}(M')={\rm LL}_{\cal B}^{[a,b]}(M)
\oplus({\rm LL}_{\cal B}^{[a,b]}(M)\otimes\mu_{-1})\\
R^{[a,b]}_{M',{\cal B}}=R^{[a,b]}_{M,{\cal B}}&&{\rm ou}&&&
R^{[a,b]}_{M',{\cal B}}=R^{[a,b]}_{M,{\cal B}}\times R^{[a,b]}_{M\otimes\mu_{-1},{\cal B}}\\
\rho^{[a,b]}_{M',{\cal B}}=\rho^{[a,b]}_{M,{\cal B}}&&{\rm ou}&&&
\rho^{[a,b]}_{M',{\cal B}}=\rho^{[a,b]}_{M,{\cal B}}\oplus \rho^{[a,b]}_{M\otimes\mu_{-1},{\cal B}}
\end{align*}
suivant que $M=M\otimes\mu_{-1}$ ou que $M\neq M\otimes\mu_{-1}$.

Si $\sigma$ est une $k_L$-repr\'esentation irr\'eductible de $K$, on note
$m^{[a,b]}_{M}(\sigma)$ la multiplicit\'e de $\sigma$ dans la r\'eduction
de $\sigma_M^{[a,b]}$.
L'\'enonc\'e suivant est une forme de la version g\'eom\'etrique de la conjecture de Breuil-M\'ezard.
\begin{prop}\label{BM1}
Les r\'eductions de ${\rm LL}_{\cal B}^{[a,b]}(M')$ et $\rho^{[a,b]}_{M',{\cal B}}$ se d\'ecomposent
sous la forme
\begin{align*}
\overline{\rm LL}_{\cal B}^{[a,b]}(M')&=\oplus_\sigma m^{[a,b]}_{M}(\sigma) I(\sigma)_{\cal B}\\
\overline{\rho}_{M',{\cal B}}^{[a,b]}&=\oplus_\sigma m^{[a,b]}_{M}(\sigma) {\bf V}(I(\sigma)_{\cal B})
\end{align*}
\end{prop}
\begin{proof}
Le premier \'enonc\'e est une cons\'equence du cor.\,\ref{GAB2}. Le second s'en d\'eduit
via la fonctorialit\'e de $\Pi\mapsto{\bf V}(\Pi)$.
\end{proof}
\begin{rema}
Si ${\cal B}$ n'est pas un bloc supersingulier,
${\bf V}(I(\sigma)_{\cal B})$ est de rang~$1$ sur un quotient de $R^{{\rm ps},\delta}_{\cal B}$, mais
ce dernier agit par un pseudo-caract\`ere de dimension~$2$ (cf.~th.\,\ref{pasku6}), pas de dimension~$1$.
\end{rema}

La conjecture de Breuil-M\'ezard (version\footnote{Il semble y avoir une certaine
latitude dans l'\'enonc\'e de la conjecture: par exemple, on peut prendre des anneaux de 
d\'eformations de repr\'esentations galoisiennes,
ou de repr\'esentations encadr\'ees; 
la version que nous obtenons utilise des anneaux de d\'eformations de pseudo-caract\`eres -- le th.\,\ref{pasku6}
identifie $Z_{\cal B}^\delta$ \`a un tel anneau.} g\'eom\'etrique~\cite{EG}) postule une d\'ecomposition
analogue pour les supports des objets ci-dessus, vus comme faisceaux sur
${\rm Spec}\,Z^{\delta}_{\cal B}$.

On note $R_{M',{\cal B}}^{[a,b],+}$ le quotient \`a travers lequel $Z_{\cal B}^\delta$ agit
sur $I(\sigma)_{\cal B}$ (cet anneau d\'epend a priori de $\sigma$, mais nous ne l'indiquons pas
sur la notation). Alors $R_{M',{\cal B}}^{[a,b],+}$ s'injecte dans $R_{M',{\cal B}}^{[a,b]}$
et on a $R_{M',{\cal B}}^{[a,b]}=R_{M',{\cal B}}^{[a,b],+}[\frac{1}{\varpi}]$.
Comme $R_{M',{\cal B}}^{[a,b]}$ est de dimension~$1$ puisque c'est un produit d'anneaux principaux,
on en d\'eduit que $R_{M',{\cal B}}^{[a,b],+}/\varpi$
est de dimension~$1$ par platitude de $R_{M',{\cal B}}^{[a,b],+}$ sur $\O_L$.

Soit ${\goth p}$ un id\'eal premier minimal de $R_{M',{\cal B}}^{[a,b],+}/\varpi$.
Notons encore ${\goth p}$ l'image inverse de ${\goth p}$ dans $Z_{\cal B}^\delta$; c'est
un id\'eal premier de $Z_{\cal B}^\delta$ et $Z_{\cal B}^\delta/{\goth p}$ est de dimension~$1$.
La conjecture de Breuil-M\'ezard postule une formule pour la longueur
du localis\'e $(R_{M',{\cal B}}^{[a,b],+}/\varpi)_{\goth p}$ comme $(Z_{\cal B}^\delta)_{\goth p}$-module
(on note $\ell_{\goth p}(M)$ la longueur d'un $(Z_{\cal B}^\delta)_{\goth p}$-module).
On la d\'eduit du r\'esultat suivant (le membre de gauche se calcule en utilisant la prop.\,\ref{BM1}):
\begin{prop}\label{BM3}
Si ${\goth p}$ est comme ci-dessus,
$$\ell_{\goth p}(\overline{\rho}_{M',{\cal B}}^{[a,b]})
=2\ell_{\goth p}((R_{M',{\cal B}}^{[a,b],+}/\varpi)_{\goth p})$$
\end{prop} 
\begin{proof}
Les longueurs ne changent pas par compl\'etion. Notons donc $Z$ le compl\'et\'e
de $(Z_{\cal B}^\delta)_{\goth p}$ pour la topologie ${\goth p}$-adique, $R$ celui
de $(R_{M',{\cal B}}^{[a,b],+})_{\goth p}$, et $V$ celui de ${\bf V}(I(\sigma)_{\cal B})$.

Alors $R/\varpi$ est un corps local de
dimension~$1$ -- c'est le corps des fractions de $(R_{M',{\cal B}}^{[a,b],+}/\varpi)/{\goth p}$ --
et donc $R$ est un anneau de valuation discrete dont $R[\frac{1}{\varpi}]$ est le corps des fractions.
Comme ${\bf V}(I(\sigma)_{\cal B})$ est sans $\varpi$-torsion puisque $I(\sigma)_{\cal B}$ 
l'est (cor.\,\ref{petit11}), on en d\'eduit que $V$ est un module libre sur $R$, de rang $2$ puisque
$\rho_{M',{\cal B}}^{[a,b]}=L\otimes_{\O_L}I(\sigma)_{\cal B}$ est libre de rang~$2$ sur $R_{M',{\cal B}}^{[a,b]}$.
Il s'ensuit que $V/\varpi$ est libre de rang~$2$ sur $R/\varpi$; d'o\`u le r\'esultat.
\end{proof}

\end{document}